\documentclass[10pt,a4paper]{amsart}
\usepackage[pdftex,colorlinks=true,
                      pdfstartview=FitV,
                      linkcolor=blue,
                      citecolor=blue,
                      urlcolor=blue,
          ]{hyperref}
\usepackage{amsfonts}
\usepackage{amssymb}
\usepackage{diagrams}
\newarrow{RLCorresponds}{harpoon}---{harpoon}
\newtheorem{thm}{Theorem}[section]
\newtheorem{lemma}[thm]{Lemma}
\newtheorem{cor}[thm]{Corollary}

\newtheorem{prp}[thm]{Proposition}
\newtheorem{df}[thm]{Definition}
\newtheorem{rem}[thm]{Remark}

\newtheorem{ex}[thm]{Example}

\newcommand{\tr}{{\rm tr}}
\newcommand{\Aut}{{\rm Aut}}
\newcommand{\Dim}{{\rm Dim}}
\newcommand{\Mod}{\mbox{{\rm Mod}}}
\newcommand{\Hom}{\mbox{{\rm Hom}}}
\newcommand{\End}{\mbox{{\rm End}}}
\newcommand{\Mor}{\mbox{{\rm Mor}}}
\newcommand{\res}{\mbox{${\rm res}_{\downarrow H}^G$}}
\newcommand{\tensind}{\mbox{${\rm Ind}_H^{\otimes G}$}}
\newcommand{\ind}{\mbox{${\rm Ind}_H^{\times G}$}}
\newcommand{\Halg}{\mbox{$\fld_H{\tt alg}$}}
\newcommand{\Galg}{\mbox{$\fld_G{\tt alg}$}}
\newcommand{\GAalg}{\mbox{$A_G{\tt alg}$}}
\newcommand{\alg}{\mbox{$\fld{\tt alg}$}}
\newcommand{\A}{{\mathcal A}}
\newcommand{\B}{{\mathcal B}}
\newcommand{\D}{{\mathcal D}}

\newcommand{\F}{{\mathcal F}}

\newcommand{\R}{{\mathcal R}}

\newcommand{\C}{\mathfrak{C}}

\newcommand{\Pp}{\mathcal{P}}

\newcommand{\ts}{\mathfrak{Ts}}
\newcommand{\fld}{{k}}
\newcommand{\id}{{\rm id}}
\newcommand{\chr}{{\rm char}}

\author{Peter Fleischmann and Chris Woodcock}
\thanks{Corresponding author Peter Fleischmann}
\address{School of Mathematics, Statistics and Actuarial Science \\
University of Kent \\
Canterbury CT2 7NF \\
United Kingdom}
\email{P.Fleischmann@kent.ac.uk}
\email{C.F.Woodcock@kent.ac.uk}

\title [Local Galois extensions]{Modular group actions on algebras and $p$-local\\ Galois extensions for
finite groups}

\begin{document}

\begin{abstract}
Let $\fld $ be a field of positive characteristic $p$ and let $G$ be a finite group.
In this paper we study the category $\ts_G$ of finitely generated commutative $\fld$-algebras
$A$ on which $G$ acts by algebra automorphisms with surjective trace.
If $A=\fld[X]$, the ring of regular functions of a variety $X$, then trace-surjective group actions on $A$
are characterized geometrically by the fact that all point stabilizers on $X$ are $p'$-subgroups
or, equivalently, that $A^P\le A$ is a Galois extension for every Sylow $p$-group of $G$.
We investigate categorical properties of $\ts_G$, using a version of Frobenius-reciprocity
for group actions on $k$-algebras, which is based on tensor induction for modules. We also describe projective generators
in $\ts_G$, extending and generalizing the investigations started in \cite{nonlin},
\cite{locmod} and \cite{universal} in the case of $p$-groups. As an application we show
that for an abelian or $p$-elementary group $G$ and $\fld$ large enough, there is always a faithful (possibly nonlinear) action on a polynomial
ring such that the ring of invariants is also a polynomial ring. This would be false for \emph{linear} group actions by a result
of Serre. If $A$ is a normal domain and $G\le\Aut_\fld(A)$ an arbitrary finite group, we show
that $A^{{\rm O}_p(G)}$ is the integral closure of $k[{\rm Soc}(A)]$, the subalgebra of $A$ generated by the
simple $\fld G$-submodules in $A$. For $p$-solvable groups this leads to a structure theorem on trace-surjective
algebras, generalizing the corresponding result for $p$-groups in \cite{nonlin}.
\end{abstract}


\maketitle
\setcounter{section}{-1}
\small
\section{Introduction} \label{Intro}

Let ${G}$ be an arbitrary finite group, $\fld$ a field and $A$ a commutative
$\fld$-algebra on which ${G}$ acts by $\fld$-algebra automorphisms;
then we call $A$ a $\fld-{G}$ algebra. By $\Galg$ we denote the category of
commutative $\fld-G$ algebras with $G$-equivariant algebra homomorphisms; if $G=1$, we set
$\fld{\tt alg}:=\fld_1{\tt alg}$ to denote the category of all commutative $\fld$-algebras.
Let $A^{G}:=\{a\in A\ |\ ag=a\ \forall g\in {G}\}$ be the \emph{ring of invariants},
the primary object of study in invariant theory.

One of the main challenges is to describe
structural properties of the ring $A^{G}$, assuming that $A$ is ``nice", for example
a Cohen-Macaulay ring or a polynomial ring.
Clearly $A^{G}$ is a subring of $A$ as well a submodule of the
$A^{G}$-module $A$ (denoted by $_{A^{G}}A$). It is easy to see that the ring extension $A^{{G}}\le A$ is integral.
If moreover $A\in\Galg$ is finitely generated as $\fld$-algebra, then so is $A^{{G}}$, by a
classical result of Emmy Noether (\cite{N2}).

Let $R$ be an arbitrary commutative ring with subring $S\le R$. A surjective homomorphism of $S$-modules,
$r:\ {_SR}\to {_SS}$, is called a
\emph{Reynolds} operator, if $r_{|S}={\rm id}_S$, or equivalently,
$r^2=r\in {\rm End}(_SR)$. The existence of a Reynolds operator is obviously
equivalent to the fact that $S$ is a direct summand of $R$ as an $S$-module.
The following well known result of Hochster-Eagon is of fundamental importance in
invariant theory: (see \cite{BH}, Theorem 6.4.5, pg. 282)
\begin{thm}(Hochster-Eagon)\label{Hochster-Eagon}
Let $R$ be a Cohen-Macaulay ring with subring $S\le R$ and Reynolds operator $r$, such
that $R$ is integral over $S$. Then
if $R$ is Cohen-Macaulay, so is $S$.
\end{thm}

Let $\tr:=\tr_{G}:\ A\to A^{G},\ a\mapsto \sum_{g\in {G}} ag$
be the \emph{transfer map} or
\emph{trace map}. This is obviously a homomorphism of $A^{G}$-modules, therefore
the image ${\rm tr}(A)\unlhd A^{G}$ is an ideal in $A^{G}$. If
${\rm tr}(a)=1$ for some $a\in A$, then for any $a'\in A$ we have
${\rm tr}(a\cdot {\rm tr}(aa'))={\rm tr}(aa')\cdot {\rm tr}(a)={\rm tr}(aa')$,
hence the map $A\to A^{{G}},\ a'\mapsto {\rm tr}(aa')$ is a Reynolds operator.

This motivates the following

\begin{df}\label{ts-alg_df}
An algebra $A\in\Galg$ such that $\tr(A)=A^{G}$ will be called a
{\bf trace-surjective} $\fld-{G}$-algebra.
With $\ts:=\ts_{G}$ we denote the full subcategory of $\Galg$ consisting
all trace-surjective algebras which are {\bf finitely generated} as $\fld$-algebras.
\end{df}

Now we obtain the following well known consequence of Theorem \ref{Hochster-Eagon}:
\begin{cor}\label{ts_cm}
Let $A\in\ts_{G}$ be a Cohen-Macaulay ring. Then so is $A^{{G}}$.
\end{cor}

An important class of $\fld-{G}$-algebras arises in the following way: Let $V\ne 0$ be a finite dimensional $\fld$-vector space, ${G}\le {\rm GL}(V)$ a finite linear group
and set $A:={\rm Sym}(V^*)=\oplus_{i=0}^\infty A_i$ the graded symmetric algebra
over the dual space $V^*$. Then $A$ is isomorphic to the polynomial ring
$\fld[X_1,\cdots,X_n]$ with $V^*=\oplus_{i=1}^n \fld X_i$, on which
${G}$ acts by the graded algebra homomorphisms which extend the dual
action on $V^*$. We will refer to this class of group actions as \emph{linear} group
actions. In that case it is easy to see that $\tr$ is surjective if and only if $p:=\chr(\fld)$ does not divide the group order
$|{{G}}|$. It $p\not|\ |G|$ then $1/|G|\cdot \tr$ is a Reynolds operator and  Corollary
\ref{ts_cm} implies the well known result that the invariant rings
${\rm Sym}(V^*)^{{G}}$ are Cohen-Macaulay rings.
If $p$ divides $|{{G}}|$, then the image ${\rm tr}(A)$ does not contain the constants
$A_0\cong \fld$, so it is a proper ideal of $A^{{G}}$. In particular if $|{G}|=p^s$,
then due to a theorem of Kemper (\cite{kem:h}),
$A^{{G}}$ can only be a Cohen-Macaulay ring if the linear group
${G}\le {\rm GL}(V)$ is generated by \emph{bireflections}, i.e. linear transformations
with fixed-point space of codimension $\le 2$ (see also
\cite{Campbell_Wehlau_book}, Theorem 9.2.2. page 170).

A similar situation arises if one asks under what conditions the property that
$A$ is a regular ring or a polynomial ring is inherited by the ring of invariants.
If $A={\rm Sym}(V^*)$ and $A^{G}$ is a polynomial ring, then,
due to a celebrated result by Serre (\cite{bou}),
${G}$ is generated by pseudo-reflections, i.e. linear transformations with
fixed point space of codimension one. If ${\rm char}(\fld)\not|\ |{G}|$,
the converse also holds by the well-known theorem of Chevalley-Shephard-Todd and Serre (see e.g. \cite{B} or
\cite{Lsb}).

In the rest of this paper, unless explicitly said otherwise, $\fld$ will be a field of characteristic $p>0$ and $G$ a finite group of order $|G|$. Then all non-diagonalizable pseudo-reflections, the \emph{transvections}, have order $p$. If $g\in {\rm Gl}(V)$ is a unipotent bireflection, then $g-1$ is nilpotent with
$(V)(g-1)=:W$ a $g$-stable subspace of dimension $\le 2$.
Therefore $(V)(g-1)^3=(W)(g-1)^2=0$, hence $(V)(g^p-1)=(V)(g-1)^p=0$ if $p>2$.
This shows that bireflections of $p$-power order have order $p$ if $p>2$.
It therefore follows from the results by Kemper, that for any finite
$p$-group $G\le {\rm Gl}(V)$ which is not generated by elements of
order $p$, the invariant ring ${\rm Sym}(V^*)^G$ is not Cohen-Macaulay, let alone
a polynomial ring.

The situation changes completely if we remove the condition that $G$ is a linear
group, and allow for only ``mildly nonlinear" actions:

\begin{ex}\label{C_4_ex}
Let $p=2$, $G=\langle g\rangle\cong C_4$,
$A=\mathbb{F}_2[x_1,x_2,x_3]$ and $g:\ x_1\mapsto x_1$,\ $x_2\mapsto x_2+x_1;$\ $x_3\mapsto x_3+x_2.$ Since $g$ is not a transvection, $A^G$ cannot be polynomial by Serre's theorem. Indeed
$A^G=\mathbb{F}_2[x_1,f_2,f_3,f_4]$ with $f_2:=x_1x_2+x_2^2$, $f_3:=x_1^2x_3+x_1x_2^2+x_1x_3^2+x_2^3$,
$f_4:=x_1^2x_2x_3+x_1^2x_3^2+x_1x_2^2x_3+x_1x_2x_3^2+x_2^2x_3^2+x_3^4$
and one relation:
$$x_1^2f_4-f_2^3-x_1f_2f_3-f_3^2=0.$$
It follows that $A^G[1/x_1]=\mathbb{F}_2[x_1^{\pm 1},f_2,f_3].$
Note that ${\rm tr}(x_1x_2x_3)=\sum_{g\in G}g(x_1x_2x_3)=x_1^3$,
hence the map
${\rm tr}:\ A[1/x_1]\to (A[1/x_1])^G\ {\rm is\ surjective}.$
Consider the ``dehomogenization":
$$D_{x_1}:=(A[1/x_1])_0\cong A/(x_1-1)A.$$
Then $D_{x_1}$ is a polynomial ring of (Krull-) dimension $2$ with faithful
non-linear $G$-action and polynomial ring of invariants
$D_{x_1}^G=\fld[f_2/x_1^2,f_3/x_1^3]$.
\end{ex}

Generalizing this example we consider for $G\le {\rm GL}(V)$,
$A:={\rm Sym}(V^*)$ and $0\ne x\in (V^*)^G$ the $\mathbb{Z}$-graded
algebra $A[1/x]$ and define the ``dehomogenization":
\[D_x:=(A[1/x])_0\cong A/(x-1)A.\]
It is known that a graded algebra and its dehomogenizations share many interesting properties (see e.g.
\cite{BH} pg. 38 and the exercises 1.5.26, 2.2.34, 2.2.35 loc. cit.)
Clearly the algebra $D_x$ is a polynomial ring of Krull-dimension $|G|-1$.

Then we have the following
\begin{lemma}\label{D_x_ts}
Assume that $x^N\in{\rm tr}_G(A)$ for some $N\in\mathbb{N}$.
Then $D_x^G$ is a Cohen-Macaulay ring.
\end{lemma}

\begin{proof}
Let $x^N={\rm tr}_G(f)$. Without loss of generality we can assume that
$f$ is homogeneous of degree $N$, hence
$a:=\frac{f}{x^N}\in D_x$ with ${\rm tr}_G(a)=1$, so $D_x\in\ts$.
It follows from Corollary \ref{ts_cm} that $D_x^G$ is Cohen-Macaulay.
\end{proof}

Despite the fact that finite $p$-groups which are not generated by
elements of order $p>2$ cannot act \emph{linearly} on polynomial rings with
Cohen-Macaulay invariants, we have the following first little observation:

\begin{cor}\label{exists_CM_inv}
Let $G$ be an arbitrary finite group and $\fld$ an arbitrary field.
Then there is always a faithful (maybe mildly non-linear) action of $G$
on a polynomial ring with Cohen-Macaualay ring of invariants.
\end{cor}
\begin{proof}
Let $V=\oplus_{g\in G}\fld X_g\cong \fld G$ be the regular module and
$x:=\sum_{g\in G} X_g=\tr_G(X_1)$. Then $D_x\in\ts_G$, in particular faithful,
with $D_x^G$ a Cohen-Macaulay ring.
\end{proof}

Of course the polynomial ring $D_x\in\ts_G$ of Corollary \ref{exists_CM_inv}
has Krull-dimension $|G|-1$, whereas example  \ref{C_4_ex} shows that
one can do better. This raises the question for the \emph{minimal} Krull-dimensions
of polynomial rings with faithful group action and Cohen-Macaualay or polynomial
rings of invariants. The latter question has been raised for $p$-groups in
\cite{nonlin}, \cite{locmod} and \cite{catgal}. In \cite{nonlin} an answer was
given for the case of the prime field $\fld=\mathbb{F}_p$. In this paper we will
generalize the methods and some results of these papers, to deal with arbitrary
finite groups of order divisible by $p$.

With regard to polynomial rings of invariants the situation is less clear. We do not know
whether for an arbitrary finite group there is always a faithful action on a polynomial
ring, such that the ring of invariants is again a polynomial ring.
Combining some results of \cite{nonlin} on trace-surjective algebras for $p$-groups with the above mentioned theorem of Serre
we obtain the following ``polynomial analogue" of \ref{exists_CM_inv} at least for abelian or $p$-elementary\footnote{i.e.
a direct product of a cyclic $p'$-group and a $p$-group} groups:

\begin{thm}\label{abelian_pol_intro}
Let $\fld$ be algebraically closed, $G$ a finite abelian or $p$-elementary group and $r$ the smallest prime divisor
of $|G|$. Then there exists a polynomial ring $\F$ of Krull-dimension $\le \log_r(|G|)$, such that $G$ acts faithfully
on $\F$ with $\F^G\cong \F$ as algebras.
\end{thm}
\begin{proof}
This is an immediate consequence of Proposition \ref{abelian_pol_inv}.
\end{proof}

Results like this convince us that the study of the category $\ts_G$ is worthwhile, not just for $p$-groups,
but for arbitrary finite groups. For example, it turns out that the category $\ts_G$ has an interesting
\emph{geometric significance}, its objects are characterized by the following ``$p$-local Galois property":

\begin{thm}\label{local_galois_geom_intro}
Let $\fld$ be an algebraically closed field of characteristic $p>0$, $X$ an affine $\fld$-variety
with ring of regular functions $A=\fld[X]$ and $G$ a finite group acting on $X$. Then the following are equivalent:
\begin{enumerate}
\item $A\in\ts_G$.
\item For every $x\in X$ the point-stabilizer $G_x$ has order coprime to $p$.
\item For one (and then every) Sylow $p$-group $P\le G$, the ring extension $A^P\le A$ is a Galois-extension in the sense of
Auslander and Goldmann \cite{AG} or Chase-Harrison-Rosenberg in \cite{chr}.
\end{enumerate}
\end{thm}
\begin{proof}
See Proposition \ref{local_galois} and Corollary \ref{local_galois_geom}.
\end{proof}

The rest of the paper is organized as follows:

In Section 1 we explain the geometric background of our results in the context of free and ``$p$-locally
free actions" of finite groups on affine varieties. Here ``$p$-locally free" means that the action restricted
to every $p$-subgroup is free. We also collect some definitions and results from \cite{nonlin}
and \cite{catgal}, which are used to prove Theorems \ref{abelian_pol_intro} and \ref{local_galois_geom_intro}, but will
also be used in later sections.\\
In Section 2 we develop the basic properties of trace-surjective algebras and also investigate categorical properties
of $\ts_G$. Although this is not an abelian category it has
``s-projective objects", which are analogues of projective modules (see Definition \ref{s-projective}), and it has (s-projective) categorical generators, which we will describe explicitly. This
generalizes definitions and results of \cite{catgal} from $p$-groups to arbitrary finite groups. In particular the special role of ``points" (i.e. ring elements with trace one) is analyzed (see Corollary \ref{ts_gen_by_pts}). As in the $p$-group case, it turns out that the dehomogenization $D_{reg}$ is a ``free generator" in the category $\ts_G$.\\
In Section 3 we discuss induction and restriction functors and the analogue of ``Frobenius reciprocity" for group actions on
commutative $k$-algebras.
Let $H\le G$ be a subgroup, then there is an obvious restriction functor $\res:\ \Galg\to\Halg$, which turns out
to have left- and right adjoints. In contrast to module theory, these adjoint functors do not coincide: in fact
the left adjoint of $\res$ is given by ``tensor induction" $\tensind:\ \Halg\to \Galg$" and the right-adjoint
is given by ordinary ``Frobenius induction" $\ind:\ \Halg\to \Galg$ (see Theorem \ref{res_adjoints}). \\
In Section 4 we apply Frobenius reciprocity in the category $\ts_G$ to investigate properties of objects that can be
detected and analyzed via restriction to Sylow $p$-groups. We prove the following analogues to well known results in module theory: An algebra $A\in\ts_G$ is s- projective if and only if its restriction ${\rm res}(A_{|P})$ is so in $\ts_P$ for a
Sylow $p$-group (see Corollary \ref{split_proj_iff}). If $B\in\ts_G$ is s- projective, then so is
$B\otimes {\rm Sym}(V)$ for any finite-dimensional $\fld G$-module $V$ (Theorem \ref{SV_sproj}). \\
In Section 5 we interpret some of the results of previous sections
as a particular version of Maschke's theorem for group actions on commutative algebras.
This can be used to describe a decomposition of tensor products of the
form $A\otimes {\rm Sym}(V)$ with $A\in\ts_G$ and $\fld G$-module $V$.
A general structure theorem on algebras $A\in\ts_G$ which was
proven in \cite{nonlin} Proposition 4.2 for $p$-groups is generalized to $p$-solvable groups (Proposition \ref{A_tens_Soc_2}).
As an application to general group actions on commutative $k$-algebras we show that
if $A\in\Galg$ is a normal domain, then $A^{{\rm O}_p(G)}$ is the
integral closure of $A_{\rm soc}$ in its quotient field. Here
$A_{\rm soc}=k[{\rm Soc}(A)]$ is the subalgebra of $A$ generated by the
simple $\fld G$-submodules contained in $A$ (Proposition \ref{A_soc_prp}).\\
The Appendix at the end of the paper contains some material on adjoint functors in a form most useful for
section 3. It has been included for the convenience of the reader and to make our exposition self-contained.

{\bf Notation}:
For a category $\frak{C}$ and objects $a,b\in\frak{C}$ we denote by
$\frak{C}(a,b)$ the set of morphisms from $a$ to $b$. The word ``ring" will always
mean ``unital ring" and the notion of a ``subring" $S\le R$ or a ``ring homomorphism" $\phi: S\to R$" will always mean ``unital subring" with $1_S=1_R$ and ``unital homomorphism" satisfying $\phi(1_S)=1_R$.
Let $G$ be a group with group ring $\fld{G}$; with $Mod-\fld{G}$ ($mod-\fld{G}$) we will denote the
category of (finitely generated) right $\fld{G}$-modules and with $\fld{G}-Mod$
($\fld{G}-mod$) we denote the corresponding categories of left modules. If $M$ is a
$\fld {G}$-bimodule, the restriction to the left or right module structure
will be indicated by $_{\fld {G}}M$ or $M_{\fld {G}}$, respectively. We will also use
standard notation from group theory, e.g. for a finite group ${G}$ and a prime $p$ we set
${\rm Syl}_p({G})$ to be the set of all Sylow $p$-groups. A ``$p'$-group" is a finite group of order
coprime to $p$, ${\rm O}_p({G}):=\cap_{P\in{\rm Syl}_p({G})} P\unlhd {G}$ the ``$p$-core" of ${G}$ and
${\rm O}_{p'}({G})\unlhd {G}$ to be the maximal normal subgroup of order coprime to $p$.
By ${\rm O}_{p,p'}({G})$ (or ${\rm O}_{p',p}({G})$, respectively) we denote the canonical
preimage of ${\rm O}_{p'}({G}/{\rm O}_p({G}))$ (or ${\rm O}_{p}({G}/{\rm O}_{p'}({G}))$).
If $\Omega$ is a set on which the group $G$ acts,
we find it useful to switch freely between ``left" and ``right"-actions, using the rule
$$\omega\cdot g:=\omega^g=\ ^{g^{-1}}\omega= g^{-1}\cdot \omega,\ \forall g\in G, \omega\in \Omega,$$
which changes a given right-$G$-action into a left one and vice versa. The set of $G$-fixed points
on $\Omega$ will be denoted by $\Omega^G$.

\section{Galois extensions and $p$-locally free group actions}

We start with some definitions and notation that will also be used later in the paper:
\begin{df}\label{stably_def} Let $R$ be a $\fld$-algebra and $n\in\mathbb{N}$.
\begin{enumerate}
\item With $R^{[n]}$ we denote the polynomial ring $R[t_1,\cdots,t_n]$ over $R$.
\item Let $\mathbb{P}=\fld[t_1,\cdots,t_m]\cong\fld^{[m]}$ and $G\le \Aut_\fld(\mathbb{P})$.
Then $\mathbb{P}$ is called {\bf uni-triangular} (with respect to the chosen generators $t_1,\cdots,t_m$),
if for every $g\in G$ and $i=1,\cdots,m$ there is $f_{g,i}(t_1,\cdots,t_{i-1})\in \fld[t_1,\cdots,t_{i-1}]$ such that $(t_i)g=t_i+f_{g,i}(t_1,\cdots,t_{i-1}).$
\item Let $m\in\mathbb{N}$, then a $\fld$-algebra $R$ is called {\bf ($m$-) stably polynomial} if
$T:=R\otimes_\fld \fld^{[m]}\cong R^{[m]}\cong \fld^{[N]}$ for some $N\in\mathbb{N}$. Assume
moreover that $R$ is a $\fld-G$ algebra and $T$ extends the $G$-action on $R$ trivially, i.e.
$T\cong R\otimes_\fld F$ with $F=F^G\cong\fld^{[m]}$. If $T$ is uni-triangular, then we call
$R$ {\bf ($m$-) stably uni-triangular}.
\item Let $V_{reg}$ be the regular representation of $G$ with dual space
$V^*_{reg}:=\oplus_{g\in G} \fld X_g\cong \fld G$, $X_g:=(X_e)g$ and
$x:=\sum_{g\in G} X_g={\rm tr}_G(X_e)\in V^G$. Then we set
$D_{reg}:=D_{reg}(G):=D_x$, the dehomogenization of ${\rm Sym}(V^*_{reg})$.
Note that $D_{reg}(G)$ is a polynomial ring in $|G|-1$ variables.
\end{enumerate}
\end{df}

The next result uses the following Theorem, which was one of the main results of
\cite{nonlin}:
\begin{thm}[\cite{nonlin} Theorems 1.1-1.3] \label{arb_p_grp_sec0}
Let $P$ be a group of order $p^n$. There exists a trace-surjective uni-triangular
$P$-subalgebra $U:=U_P\le D_{reg}$, such that $U\cong \fld^{[n]}$ is a retract of
$D_{reg}$, i.e. $D_{reg}=U\oplus I$ with a $P$-stable \emph{ideal} $I\unlhd D_{reg}$.
Moreover: $U^P\cong \fld^{[n]}$ and $D_{reg}^P\cong \fld^{[|P|-1]}$.
\end{thm}

Let $H\le {\rm GL}(V)$ be a finite subgroup with polynomial ring of invariants
$A^H={\rm Sym}(V^*)^H$, (so $H$ must be generated by pseudo-reflections) and let $P$ be
an arbitrary finite $p$-group. Then the direct product $H\times P$ acts faithfully on
the polynomial ring $\F:=A\otimes_\fld U_P$ with ring of invariants
$\F^{H\times P}\cong A^H\otimes_\fld U_P^P$, which is again a polynomial ring.
This applies to any $H\le {\rm GL}(V)$ of order coprime to $p$, which is generated by pseudo-reflections.

\begin{prp}\label{abelian_pol_inv}
Let $H$ be an abelian $p'$-group of exponent $e$, $P$ an arbitrary finite
$p$-group, $G=H\times P$ and $r$ the minimal prime divisor of $|G|$. Assume that $\fld$ contains
a primitive $e$'th root of unity, then there exists a polynomial ring $\F$ of Krull-dimension
$d\le \log_r(|H|)+\log_p(|P|)\le \log_r(|G|)$, such that $G:=H\times P$ acts faithfully on $\F$ with $\F^G\cong \fld^{[d]}$.
\end{prp}
\begin{proof} Let $H\cong\prod_{i=1}^s C_{d_i}$ with elementary divisors $1<d_1\ |\ d_2\ |\ \cdots\ |\ d_s$
and let $\eta\in\fld$ be a primitive $d_s$'th root of unity. Then
every factor $C_{d_i}$ acts on the one dimensional space $\fld$ with
generating pseudo-reflection of eigenvalue $\eta^{d_s/d_i}$.
It follows that $H$ acts on $V:=\fld^s$ as a linear group generated by pseudo-reflections.
Since $r^s\le \prod_{i=1}^s d_i=|H|$ and the polynomial ring $U_P$ has
Krull-dimension $\log_p(|P|)$, we can choose $\F$ to be
${\rm Sym}(V^*)\otimes_\fld U_P$.
\end{proof}

We are now going to explain the geometric significance of the category $\ts_G$ of trace surjective $\fld-G$ algebras:

Set $B:=A^G$ and define $\Delta:=G\star A=A\star G:=\oplus_{g\in G}d_gA$ to be the crossed product of
$G$ and $A$ with $d_gd_h=d_{gh}$ and $d_ga=g(a)\cdot d_g=(a)g^{-1}\cdot d_g$ for
$g\in G$ and $a\in A$. Let $_BA$ denote $A$ as left $B$-module, then there is a
homomorphism of rings
$$\rho:\ \Delta\to \End(_BA),\ ad_g\mapsto \rho(ad_g)=(a'\mapsto a\cdot g(a')=a\cdot (a')g^{-1}).$$
One calls $B\le A$ a \emph{Galois-extension} with group $G$ if $_BA$ is finitely
generated projective and $\rho$ is an isomorphism of rings. This definition goes back
to Auslander and Goldmann \cite{AG} (Appendix, pg.396)
and generalizes the classical notion of Galois field extensions. It also applies
to non-commutative $k-G$ algebras, but if $A$ is commutative, this definition of
`Galois-extension' coincides with the one given by Chase-Harrison-Rosenberg in \cite{chr},
where the extension of commutative rings $A^G\le A$ is called a Galois-extension
if there are elements $x_1,\cdots,x_n$, $y_1,\cdots,y_n$ in $A$ such that
\begin{equation}\label{chr_def}
\sum_{i=1}^n x_i(y_i)g=\delta_{1,g}:=
\begin{cases}
1& \text{if}\ g=1\\
0&\text{otherwise}.
\end{cases}
\end{equation}

In \cite{chr} the following has been shown:
\begin{thm}(Chase-Harrison-Rosenberg)\cite{chr}\label{chase_harrison_rosenberg}
$A^G\le A$ is a Galois extension if and only if
for every $1\ne \sigma\in G$ and maximal ideal ${\rm p}$ of $A$ there
is $s:=s({\rm p},\sigma)\in A$ with $s-(s)\sigma\not\in{\rm p}$.
\end{thm}

Now, if $X$ is an affine variety over the algebraically closed field $\fld$,
with $G\le {\rm Aut}(X)$ and $A:=\fld[X]$ (the ring of regular functions),
then for every maximal ideal ${\rm m}\unlhd A$, $A/{\rm m}\cong \fld$. Hence if $({\rm m})g={\rm m}$, then
$a-(a)g\in {\rm m}$ for all $a\in A$. Therefore we conclude

\begin{thm}\label{aff act_and_gal thm1}
The finite group $G$ acts freely on $X$ if and only if $k[X]^G\le k[X]$ is a Galois-extension.
\end{thm}

If $B\le A$ is a Galois-extension, then it follows from equation (\ref{chr_def}),
that ${\rm tr}(A)=A^G=B$ (see \cite{chr}, Lemma 1.6), so $A\in \ts_G$.
It also follows from Theorem \ref{chase_harrison_rosenberg},
that for a $p$-group $G$ and $\fld$ of characteristic $p$,
the algebra $A$ is trace-surjective if and only if $A\ge A^G=B$ is a Galois-extension
(see \cite{nonlin} Corollary 4.4.). Due to a result of Serre, the only finite groups acting
freely on $\mathbb{A}^n$ are finite $p$-groups (see \cite{serre_how_to} or \cite{catgal} Theorem 0.1).
Using this we obtain
\begin{cor}\label{ts_gal1}
Let $\fld$ be algebraically closed. Then the finite group $G$ acts freely on $X\cong \mathbb{A}^n$ if and only if
$G$ is a $p$-group with $p={\rm char}(\fld)$ and $\fld[X]\in\ts_G$.
\end{cor}

Since for $p$-groups in characteristic $p$ the trace-surjective algebras coincide with
Galois-extensions over the invariant ring, we obtain from Theorem
\ref{aff act_and_gal thm1}:

\begin{cor}\label{intro_cor_1}
If $\fld$ is an algebraically closed field of characteristic $p>0$, $X$ an affine $\fld$-variety
with $A=\fld[X]$ and $G$ a finite $p$-group, then the following are equivalent:
\begin{enumerate}
\item $G$ acts freely on $X$;
\item $A^G\le A$ is a Galois extension;
\item $A\in\ts_G$.
\end{enumerate}
\end{cor}

For an arbitrary finite group $G$ the properties $A\in\ts_G$ and $A^G\le A$ Galois are \emph{not} equivalent. Indeed,
if $1<|G|$ is coprime to $p=\chr(\fld)$, then $A\in\ts_G$, but $A^G\le A$ may not be Galois.
In fact the following holds, regardless whether $p$ divides $|G|$ or not:

\begin{prp}\label{SymV_never_galois_for non_triv_G}
Let $A$ be an $\mathbb{N}_0$-graded, connected, noetherian normal domain
and assume that $G\le {\rm Aut}(A)$ is a finite group of graded automorphisms
(e.g. $A={\rm Sym}(V^*)$ with $G\le {\rm GL}(V)$). Then
$A^G\le A$ is Galois if and only if $G=1$.
\end{prp}
\begin{proof}
Let $B:=A^G$; it follows from \cite{catgal} Proposition 1.5
that $A^G\le A$ is Galois if and only if $_BA$ is projective and
$A=\D_{A,B}$, the Dedekind different, which in this case coincides with
the homological different $\D_{A,B,hom}:=\mu({\rm ann}_{A\otimes_BA}(J)).$
Here $\mu:\ A\otimes_BA\to A$ is the multiplication map with kernel $J$.
By the assumption, $1_A\in \D_{A,B,hom}$, so
$1_{A\otimes_B A}-x\in {\rm ann}_{A\otimes_BA}(J)$ for some $x\in J$
and for every $j\in J$ we get $j=xj$, hence $J=J^2$. Since $(A\otimes_B A)_0\cong \fld$ and
$J< A\otimes_B A$ is a proper ideal, $J\cap (A\otimes_B A)_0=0$ so $J\le (A\otimes_B A)_+$
and the graded Nakayama lemma yields $J=0$.
Now $_BA$ is a reflexive $B$-module with $A\otimes_B A\cong A$.
Let $i:\ B\hookrightarrow A$ be the canonical embedding
and ${\rm p}\in {\rm spec}_1(B)$, then $B_{\rm p}$ is a discrete valuation ring,
hence$_{B_{\rm p}}A_{\rm p}$ is f.g. free in $B_{\rm p}$-mod of rank
$n$, say. We get $A_{\rm p}\cong B_{\rm p}^n\cong A_{\rm p}\cong$
$B_{\rm p}^n\otimes_{B_{\rm p}} B_{\rm p}^n\cong B_{\rm p}^{n^2}$, so $n=n^2=1$ and $B_{\rm p}\cong A_{\rm p}$. But
$i_{\rm p}(B_{\rm p})\le A_{\rm p}$ are both integrally closed
in $\mathbb{L}:={\rm Quot}(B_{\rm p})={\rm Quot}(A_{\rm p})$, hence
$A_{\rm p}={\rm int.closure}_\mathbb{L}(i_{\rm p}(B_{\rm p}))=
i_{\rm p}(B_{\rm p})$ so
$i_{\rm p}:\ B_{\rm p}\hookrightarrow A_{\rm p}$ is an isomorphism
and $i$ is a pseudo-isomorphism between reflexive $B$-modules. Therefore $i$ is an isomorphism.
Now it follows from standard Galois theory that $G=1$.
\end{proof}

Let $A\in\Galg$ and $Q\in{\rm spec}(A)$ a prime ideal with $q:=Q\cap A^G\in{\rm spec}(A^G)$
and residue class fields $k(q)={\rm Quot}(A^G/q)\le k(Q):={\rm Quot}(A/Q)$.
Then one defines the \emph{inertia group}
$$I_G(Q):=\{g\in G\ |\ a-(a)g\in Q\ \forall\ a\in A\}.$$
It is well known that $I_G(Q)\unlhd G_Q:={\rm Stab}_G(Q)$ with
$G_Q/I_G(Q)={\rm Aut}_{k(q)}(k(Q))$. The following result generalizes Corollary \ref{intro_cor_1}, showing
that $A\in\ts_G$ if and only if $A^G\le A$ is a $p$-local Galois-extension:

\begin{prp}\label{local_galois}
Let $A\in\Galg$ then the following are equivalent:
\begin{enumerate}
\item $A\in\ts_G$;
\item For some (any) Sylow $p$-subgroup $P\le G$, $A_{|P}\in \ts_P$.
\item For every $1\ne g\in G$ of order $p$ and all $\frak{m}\in{\rm max-spec}(A)$ there is
$a\in A$ with $a-(a)g\not\in\frak{m}$.
\item $I_G(Q)$ is a $p'$-group for every $Q\in{\rm spec}(A)$.
\end{enumerate}
\end{prp}
\begin{proof}
``(1) $\iff$ (2)":\ This follows from Lemma \ref{ts_restrict} and the fact that all Sylow $p$-groups
are conjugate. \\
``(2) $\iff$ (3)":\ By Theorem \ref{chase_harrison_rosenberg} and the conjugacy of Sylow groups,
condition (2) is equivalent to (3) with ``order $p$" replaced by ``order $p^m$ for some $m$". But
if $g\in G$ has order $p^m$ with $m>1$, then $g^{p^{m-1}}$ has order $p$ and
$a-(a)g^{p^{m-1}}=$  $(a-(a)g)+((a)g-(a)g^2)+((a)g^2-(a)g^3)+\cdots+ ((a)g^{p^{m-1}-1}-(a)g^{p^{m-1}}).$
So if $a-(a)g^{p^{m-1}}\not\in\frak{m}$, some $(a)g^i-(a)g^{i+1}\not\in\frak{m}$ also.\\
``(3) $\iff$ (4)":\ Obviously (3) is equivalent to (4) if the $Q$'s are maximal ideals. Since
$Q\le Q'\in {\rm spec}(A)$ implies $I_G(Q)\le I_G(Q')$ the claim follows.
\end{proof}

\begin{cor}\label{local_galois_geom}
If $\fld$ is an algebraically closed field of characteristic $p>0$, $X$ an affine $\fld$-variety
with $A=\fld[X]$ and $G$ a finite group, then the following are equivalent:
\begin{enumerate}
\item For every $x\in X$ the point-stabilizer $G_x$ has order coprime to $p$;
\item $A\in\ts_G$.
\end{enumerate}
\end{cor}

\section{Basic observations on trace-surjective $\fld -G$ algebras}

In the following we will recall some well known results from
representation theory of finite groups, which in many textbooks are
formulated and proved for \emph{finitely generated} modules over artinian
rings or algebras. In view of our applications we need to avoid those restrictions,
so we include short proofs of some of these results, whenever we need
to establish them in a more general context.

Let $R$ be a ring and $M$ an $R$-module. Then the
{\bf socle} of $M$ is the sum of all simple submodules, hence the
unique maximal semisimple submodule of $M$ and denoted by ${\rm Soc}(M)$.
We start with the following elementary observation:

\begin{lemma} \label{observation1}(\cite{nonlin} Lemma 2.1)
Let $I$ be an index set and $W$ a (left) $R$ - module with
submodules $V, V_i$ for $i\in I$. Then the following hold:
\begin{enumerate}
\item If $\sum_{i\in I} V_i = \oplus_{i\in I} V_i$ is a direct sum
in $W$, then
$${\rm Soc}(V)\cap (\oplus_{i\in I} V_i) = {\rm Soc}(V)\cap (\oplus_{i\in I} {\rm
Soc}(V_i)) = V\cap (\oplus_{i\in I} {\rm Soc}(V_i)).$$
\item Assume that ${\rm Soc}(V_i)\le V_i$ is an essential extension for every $i\in I$
(e.g. if $_RR$ is artinian), then we have
$$\sum_{i\in I} V_i = \oplus_{i\in I} V_i
\iff \sum_{i\in I} {\rm Soc}(V_i) = \oplus_{i\in I} {\rm
Soc}(V_i).$$
\end{enumerate}
\end{lemma}

Now let $\fld $ be a field, $G$ a finite group and $V$ a (left) $\fld G$ -
module. For any subgroup $H\le G$, we denote by $V^H$ the space of
$H$ - fixed points in $V$ and define the (relative) transfer map
$$t_H^G:\ V^H \to V^G,\ v \mapsto \sum_{g\in G\backslash H}\ ^gv,$$
where $G\backslash H$ is a system of coset representatives such that
$G = \uplus_{g\in G\backslash H}\ gH.$ If $W$ is another left
$\fld G$-module, then $G$ has a natural (left) action on ${\rm
Hom}_k(V,W)$ by conjugation, i.e. $^g\alpha(v) =
g(\alpha(g^{-1}(v)))$ with $${\rm Hom}_k(V,W)^H = {\rm
Hom}_{kH}(V,W).$$ Note that $t_H^G(\alpha) = \sum_{x\in G\backslash
H}\ ^x\alpha\in {\rm Hom}_{kH}(V,W)$.

The following is D. Higman's criterion for \emph{relative
$\fld H$-projectivity} of a $\fld G$-module:

\begin{prp}\label{higman}(\cite{nonlin} Proposition 2.2.)
Let $V$ be a $\fld G$-module, then the following are equivalent:
\begin{enumerate}
\item There is $\alpha\in {\rm End}_{kH}(V)$ with
$t_H^G(\alpha)={\rm id}_V$.
\item $V$ is a direct summand of $\fld G\otimes_{kH}V.$
\end{enumerate}
A module $V$ satisfying one of these equivalent conditions is called
{\bf relatively $H$-projective}.
\end{prp}

\begin{rem}\label{proj}
Note that $\fld G\otimes_kV \cong \oplus_{i\in I} kG^{(i)}$, if $V\cong
\oplus_{i\in I} k^{(i)}$ as a $\fld $-space. Hence $V$ is relatively
$1$-projective, if and only if $V$ is a summand of a free
$\fld G$-module, i.e. if and only if $V$ is projective.
\end{rem}

Let $P$ be a finite $p$ - group and $\fld $ have characteristic $p$.
The following lemma is well known for finitely generated
$\fld G$-modules, but is true in general (see \cite{nonlin} Lemma 2.3):

\begin{lemma} \label{rec_free}
For any $V\in Mod-\fld P$ the following are
equivalent:
\begin{enumerate}
\item
$t_1^P(V) \ne 0$;
\item there is a free direct summand $0\ne F\le V$ containing $t_1^P(V)$.
\end{enumerate}
Moreover $V$ is free if and only if $t_1^P(V) = V^P$. If $v\in V$
satisfies $t_1^P(v)\ne 0$, then
$\langle vg\ |g\in P\rangle \in mod-\fld P-mod$ is free of rank one.
\end{lemma}

In the following $\fld $ is a field of characteristic $p\ge 0$ and $G$ is a
finite group. A $\fld $-algebra $R$ will be called a $\fld-G$ algebra, if
$G$ acts on $R$ by $\fld $-algebra automorphisms. This renders $R$
a $\fld G$-module. With $\Galg$ we denote the category of \emph{commutative}
$\fld-G$ algebras with $G$-equivariant algebra homomorphisms and we set
$\fld{\tt alg}:=\fld_1{\tt alg}$ to denote the category of all commutative $\fld$-algebras.

If this $\fld G$-module is trace-surjective, then $R$ is called
a {\bf trace-surjective $\fld -G$ algebra}.

\begin{lemma}\label{tr}
Let $R$ be a $\fld -G$-algebra and $H\le G$ a subgroup then the
following are equivalent:
\begin{enumerate}
\item $1=t_H^G(r)$ for some $r\in R$;
\item $R^G=t_H^G(R^H)$;
\item $R$ is relatively $H$-projective as a $\fld G$-module.
\end{enumerate}
\end{lemma}
\begin{proof} Clearly (1) $\iff$ (2), since $t_H^G(R^H)$ is a two-sided ideal
of $R^G$. \\
For $r\in R$ let $\mu_r\in {\rm End}_k(R)$ denote the
homomorphism given by left-multiplication, i.e. $\mu_r(s)=r\cdot s$
for all $s\in R$. Then
$$^g(\mu_r)(s) = g(\mu_r(g^{-1}s)) = g(r\cdot g^{-1}s)=(gr)\cdot s =
\mu_{gr}(s),$$ hence the map
$$\mu:\ R\to {\rm End}_k(R),\ r\mapsto \mu_r$$ is a unital homomorphism of $\fld -G$-algebras.
On the other hand the map
$$e:\ {\rm End}_k(R) \to R,\ \alpha\mapsto
\alpha(1)$$ satisfies $$e(^g\alpha) = g(\alpha(g^{-1}1)) =
g\alpha(1)=g\cdot e(\alpha),$$ hence it is a homomorphism of
$\fld G$-modules with $e({\rm id}_R)=1$. We have $e\circ \mu = {\rm
id}_R$ and $(\mu\circ e)({\rm id}_R) = {\rm id}_R$. If $1=t_H^G(r)$,
then ${\rm id}_R=\mu(1)=\mu(t_H^G(r)) = t_H^G(\mu(r))$. On the other
hand if ${\rm id}_R = t_H^G(\alpha)$, then $1= e({\rm id}_R) =
t_H^G(e(\alpha))$. It now follows from Lemma \ref{higman} that (1) and (3)
are equivalent.
\end{proof}

\begin{thm}\label{main}
Let $R\ne 0$ be a $\fld -G$ - algebra. Then the following are equivalent:

\begin{enumerate}
\item[(i)] $1 = t_1^G(r)$ for some $r\in R$.
\item[(ii)] $R^G = t_1^G(R).$
\item[(iii)] $R$ is a trace-surjective $\fld -G$-algebra.
\item[(iv)] There is a $\fld G$ -submodule $W\le R$, isomorphic
to the projective cover $P(k)$ of the trivial $\fld G$-module, such that
$1_R\in W$.
\end{enumerate}

\noindent Assume that one of these conditions is satisfied. Let
$\{r_i\ |\ i\in I\}$ be a $\fld $-basis of the ring of invariants $R^G$
and $\{w_j\ |\ j=1,\cdots,s\}$ a basis of $W\le R$ (with $1\in W
\cong P(k)$). Then the following hold:
\begin{enumerate}
\item For every $0\ne r\in R^G$ we have $rW \cong W\cong P(k)$.
\item $R=R^G\cdot W \oplus C$ with $R^G\cdot W = \oplus_{i\in I}r_i\cdot W = \oplus_{j=1}^sR^G\cdot w_j$
and $C$ is a projective $\fld G$-module not containing a summand $\cong
P(k)$. In particular, $R^G\cdot W$ is a free $R^G$-module.
\item For every $G$-stable proper left ideal $J\unlhd\ _RR$ the left $\fld G$-module $R/J$ is projective
and we have $(R/J)^G \cong R^G/J^G$. For every $G$-stable two-sided proper
ideal $I\unlhd R$ the quotient ring $\bar R:= R/I$ is again a
trace-surjective $\fld -G$-algebra.
\end{enumerate}
\end{thm}

\begin{proof} The equivalence of (i),(ii) and (iii) follows from Lemma \ref{tr}.
Assume that (i) holds and let $P(k)\cong kG\epsilon$ with $\epsilon
= \epsilon^2$ a primitive idempotent in $\fld G$. Let $G^+:= \sum_{g\in
G} g = t_1^G(1_G)\in kG$, then $G^+\epsilon=G^+$ and the
$\fld G$-homomorphism $\theta:\ kG \to R$ defined by $g\mapsto g\cdot r$
maps $G^+$ to $1_R$. It follows that $t_1^G(\theta(\epsilon)) =
\theta(G^+\epsilon) = \theta(G^+)=1_R$. Since $\fld GG^+\epsilon=kG^+=
{\rm Soc}(P(k))$ is mapped to $\fld \cdot 1_R$, we see that
$\theta_{|P(k)}$ is injective and $1_R\in W:= \theta(P(k))\cong P(k).$
Assume that (iv) holds, then $1_R \in W^G = {\rm Soc}(W)=t_1^G(W)$,
hence (i) holds. We assume now that (iv) holds and will prove
statements (1)-(3).\\
(1):\ Since $W\cong P(k)$, we have $W^G \cong k$ and we can choose
the basis $\{w_i\ |\ i=1,\cdots,s\}$ such that $1 = \sum_{g\in G}\
^gw_1 = t_1^G(w_1)$. Assume $r\in R^G$ such that $rW\not\cong W$,
then $rW\cong W/X$ with $0\ne X\le W$, so $W^G=t_1^G(W)\le X$ and
$t_1^G(W/X) \le W^GX/X \le X/X=0$. Hence
$$r = r\cdot 1 = r\cdot t_1^G(w_1) = t_1^G(r\cdot w_1) =0.$$
(2):\ Since for all $i$ we have $r_iW \cong W$, ${\rm
Soc}(r_iW)=k\cdot r_i$. It follows from Lemma \ref{observation1}, that
$R^G\cdot W = \oplus_{i\in I}k\cdot r_i\cdot W$ with each $r_i\cdot
W \cong P(k)$ and again this is an injective module by H. Bass'
theorem. Hence $R = R^G\cdot W \oplus C$ with some complementary
projective $\fld G$ - module $C$. However $C^G \le R^G \le R^G\cdot W$,
since $1\in W$ and therefore $C^G=0$. Thus $C$ is a projective
$\fld G$-module not containing a summand isomorphic to $P(k)$, as required.\\
(3):\ Let $J\unlhd R$ be a $G$-stable left-ideal. Then ${\rm
id}_{|R/J} = (\mu_1)_{|R/J} = t_1^G(\mu_r)_{|R/J},$ hence $R/J$ is a
projective $\fld G$-module by \ref{higman}. Let $\bar x\in (R/J)^G$,
then
$$\bar x = 1\cdot \bar x = t_1^G(r)\bar x = \sum_{g\in G} gr\cdot
\bar x = \sum_{g\in G} gr\cdot g\bar x =$$ $$\sum_{g\in G}
g(\overline{rx}) = t_1^G(\overline{rx}) = \overline{t_1^G(rx)} \in
R^G/J^G.$$ Hence $R^G/J^G = (R/J)^G$. The last claim is obvious,
since $t_1^G(\bar r) = \sum_{g\in G} \bar r^g = \bar 1.$ \end{proof}

\begin{prp}\label{gen_by_pts}
Let $\fld $ be a field and $G$ a finite group, then the following holds:\\
Every trace-surjective $\fld -G$ algebra $R$ is generated as $\fld $-algebra by
is elements of trace $1$ if and only $0<{\rm char}(k)=p\ |\ |G|$.
\end{prp}
\begin{proof} If ${\rm char}(k)=0$ or $0<{\rm char}(k)=p$ does not divide $|G|$, then the
polynomial ring $R:=k[T]$ in one variable with trivial $G$-action is trace-surjective.
Certainly the unique element of trace one, namely $1/|G|\in k$, does
not generate that ring. Now suppose $0<{\rm char}(k)=p$ divides $|G|$, let
$\lambda\in R$ be of trace one and $r\in R^G$. Then
$\mu:=\lambda+r$ satisfies ${\rm tr}_G(\mu)={\rm tr}_G(\lambda)+|G|\cdot r=1$, so
$r=\mu-\lambda$, hence $R^G\subseteq B:=k\langle s\in R\ |\ {\rm tr}_G(s)=1\rangle$
and therefore $R^G=B^G$. Let $a\in R$ be an arbitrary element, then
$a=(a-\lambda({\rm tr}_G(a)-1))+\lambda({\rm tr}_G(a)-1)$.
Since ${\rm tr}_G(a-\lambda({\rm tr}_G(a)-1))=$
${\rm tr}_G(a)-{\rm tr}_G(\lambda)({\rm tr}_G(a)-1)=1$, we have
$a-\lambda({\rm tr}_G(a)-1)\in B$ and
$\lambda({\rm tr}_G(a)-1)\in \lambda\cdot R^G\subseteq B$, so $a\in B$ and therefore $R=B$.
\end{proof}

From now on we assume that $\fld $ is a field of characteristic $p>0$. With $\ts$ or
${\ts}_G$ we denote the category consisting of \emph{commutative} trace surjective
$\fld-G$-algebras. If $R\in \ts$ and $r\in R$ satisfies ${\rm tr}_G(r)=1$, then we call $r$ a ``point" in $R$ and denote with $\Pp_R$ the set of all points in $R$. With $\ts^o$ or $\ts_G^o$ we denote the class of algebras $\ts_G$ which are generated by
points. Thus we have

\begin{cor}\label{ts_gen_by_pts}
Let $G$ be a finite group of order divisible by $p$. Then $\ts=\ts^o$, in other
words, every $R\in \ts_G$ is generated by its points.
\end{cor}

For an arbitrary category $\C$ an object $u\in\C$ is called \emph{weakly initial},
if for every object $c\in\C$ the set $\C(u,c):=\Mor_\C(u,c)$ is not empty, i.e.
if for every object in $\C$ there is at least one morphism from $u$ to that object.
(If moreover $|\C(u,c)|=1$ for every $c\in\C$, then $u$ is called an \emph{initial object}
and is uniquely determined up to isomorphism.)
For $a,b\in\C$ one defines $a\prec b$ to mean that there is a monomorphism
$a\hookrightarrow b\in\C$ and $a\thickapprox b$ if $a\prec b$ and $b\prec a$.
According to this definition, an object $b\in\C$ is called \emph{minimal} if
$a\prec b$ for $a\in\C$ implies $b\prec a$ and therefore $a\thickapprox b$.
Clearly ``$\thickapprox$" is an equivalence relation on the object class of $\C$.

\begin{df}\label{universal_df}
Let $B\in \ts_G$, then
\begin{enumerate}
\item $B$ is called {\bf universal}, if it is a weakly initial object in $\ts_G$.
\item $B$ is called {\bf basic} if it is universal and minimal.
\item $B$ is called {\bf cyclic} if it is generated by the $G$-orbit
of one point, or equivalently, if $B\cong D_{reg}/I$ for some $G$-stable ideal
$I\unlhd D_{reg}$.
\item $B$ is called {\bf standard}, if it is a \emph{retract} of $D_{reg}$, or
in other words, if $D_{reg}=B\oplus J$, where $J\unlhd D_{reg}$ is some $G$-stable ideal.
\end{enumerate}
\end{df}

Let $A\in\ts$ and $a\in A$ be a point, i.e. $\tr(a)=1$.  Then the map $X_g\mapsto (a)g$ for $g\in G$ extends to a $k$-algebra homomorphism ${\rm Sym}(V_{reg}^*)\to A$ with $x\mapsto 1$, where
$x=\tr_G(X_1)$. Hence it defines a unique morphism $\phi:\ D_{reg}\to A$, mapping $x_g\mapsto (a)g$. In other words
$D_{reg}$ has a ``free point" $x_e$, which can be mapped to any point $a\in A\in \ts$
to define a morphism $\phi\in \ts(D_{reg},A)$. Moreover, if
$\beta:\ \frak S\to D_{reg}$ is a morphism in $\ts$,
then $\phi\circ\beta\in\ts(\frak S,A)$, so
$\frak S$ is weakly initial. On the other hand, if $\frak W\in \ts$ is weakly initial,
then there is a morphism $\alpha:\ \frak W\to D_{reg}\in \ts$, hence

\begin{prp}\label{universal_objects}
The universal objects in $\ts$ are precisely the trace-surjective
$k-G$ algebras which map to $D_{reg}$.
\end{prp}

Now let $\C$ be the category $\ts$; the following Lemma characterizes
types of morphisms by their action on points. We have:

\begin{lemma}\label{surj_surjects_points}
For $\theta\in \ts(R,S)$ let $\theta_\Pp$ denote the induced map from the set of points of $R$
to the set of points of $S$.
\begin{enumerate}
\item If $\theta$ is surjective (injective, bijective), then so is $\theta_\Pp$.
\item If $S$ is generated by points and $\theta_\Pp$ is surjective, then so is $\theta$.
\item If $p$ divides $|G|$, $\theta$ is surjective (injective, bijective) if and only if $\theta_\Pp$ is.
In particular $\theta$ is a monomorphism if and only if $\theta$ is injective.
\end{enumerate}
\end{lemma}
\begin{proof} (1)+(2):\ Assume $\theta$ is surjective. Let $s\in S$ with $\tr(s)=1$ and $r\in R$ with $\theta(r)=s$. Then $r':=\tr(r)-1\in \ker(\theta)\cap R^G$. Let $w\in R$ with $\tr(w)=1$, then $r'=\tr(r'w)$ and $v:=r-r'w$
satisfies $\theta(v)=s$ and $\tr(v)=1$, hence $\theta_\Pp$ is surjective. If $S$ is generated by points, the reverse conclusion follows.\\
(3):\ Since $p$ divides $|G|$, $S$ is generated by points, hence the claim about surjectivity follows from (1) and (2).
Now we can assume that $\theta_\Pp$ is injective and show that $\theta$ is injective. Let $w\in R$ be a point and $r,r'\in R^G$ with $\theta(r)=\theta(r'),$
then $\tr(r+w)=\tr(w)=1=\tr(r'+w)$ and $\theta(r+w)=\theta(r'+w)$, so $r+w=r'+w$ and
$r=r'$. Hence the induced map on the rings of invariants is injective.
Now let $c_i\in R$ be arbitrary with $\theta(c_1)=\theta(c_2)$. Choose $\lambda\in \Pp_R$,
then the proof of Proposition \ref{gen_by_pts} shows that
$c_i=p_i+\lambda\cdot b_i$ with $p_i\in\Pp_R$ and $b_i\in R^G$ for $i=1,2$.
Hence $w:=\theta(p_1-p_2)=\theta(\lambda)\cdot (\theta(b_2)-\theta(b_1))$ and
$$\tr(w)=0=\tr(\theta(\lambda))\cdot (\theta(b_2)-\theta(b_1))=1\cdot
(\theta(b_2)-\theta(b_1)).$$
It follows that $b_1=b_2$, $\theta(p_1)=\theta(p_2)$ hence $p_1=p_2$ and $c_1=c_2$.
For the last claim in (3), it is clear that an injective morphism is a monomorphism, so assume now
that $\theta$ is a monomorphism. It suffices to show that $\theta$ is injective on the points of $R$, so
let $a_1, a_2\in R$ be points with $\theta(a_1)=\theta(a_2)$. Define $\psi_i:\ D_{reg}\to R$ as the morphisms
determined by the map $D_{reg}\ni x_e\mapsto a_i$, then $\theta\circ\psi_1=\theta\circ\psi_2$, hence
$\psi_1=\psi_2$ and $a_1=a_2$. This finishes the proof.
\end{proof}

In an arbitrary category $\C$ an object $x$ is called ``projective" if the covariant
representation functor $\C(x,?):=\Mor_\C(x,?)$ transforms epimorphisms into
surjective maps. If $\C$ is the module category of a ring, then a morphism
is an epimorphism if and only if it is surjective, therefore a module $M$
can be defined to be projective, if $\Mor_\C(M,?)$ turns surjective morphisms to
surjective maps. In the category $\ts$, however, there are
non-surjective epimorphisms. This is the reason for the following
\begin{df}\label{s-projective}
Let $\C$ be a category of sets. We call $p\in\C$ an {\bf s-projective}
object if the covariant representation
functor $\C(p,*)$ transforms {\bf surjective morphisms} into surjective maps.
Similarly we call $i\in\C$ an {\bf i-injective} object if the contravariant representation
functor $\C(*,i)$ transforms {\bf injective morphism} into surjective maps.
\end{df}

\begin{lemma}\label{D_is_s_proj}
The algebra $D_{reg}\in\ts$ is s-projective.
\end{lemma}
\begin{proof}
Let $\theta\in \ts(A,B)$ be surjective and $\phi\in\ts(D_{reg},B)$.
Then by \ref{surj_surjects_points} $\phi(x_e)=\theta(\gamma)$
for a point $\gamma\in \Pp_A$. The map $x_e\mapsto\gamma$ extends
to a morphism $\psi\in\ts(D_{reg},A)$ with $\theta\circ\psi=\phi$.
\end{proof}

Let $\C$ be an arbitrary category. Then an object $m \in \C$
is called a \emph{generator} in $\C$, if the covariant morphism - functor
${\rm Mor}_\C(m,*)$ is injective on morphism sets. In other words, $m$ is a
generator if for any two objects $x,y \in \C$ and morphisms
$f_1,\ f_2 \in \C(x,y)$, $f_1 \ne f_2$ implies $(f_1 )_* \ne (f_2)_*$, i.e.
there is $f \in \C(m,x)$ with $f_1\circ f \ne f_2 \circ f$. It follows that
$\C(m,x)\ne\emptyset$ whenever $x\in\C$ has nontrivial automorphisms. So if every
object $x\in\C$ has a nontrivial automorphism, then generators in $\C$ are weakly initial objects.
If $\C=\ts$, then right multiplication with any $1\ne z\in Z(G)$ is a nontrivial automorphism for every object, hence if $Z(G)\ne 1$, then every generator in $\ts$ is universal.

Note that the category $\fld{\tt alg}$ of commutative $\fld$-algebras is not abelian, but it has
finite products and coproducts given by the cartesian product
$\prod_{i=1}^\ell A_i=\times_{i=1}^\ell A_i$ and the
tensor product
$\coprod_{i=1}^\ell A_i=A_1\otimes_\fld A_2\otimes_\fld\cdots\otimes_\fld A_\ell$ for $A_i\in\fld{\tt alg}$.
(In the following we will write $\otimes$ instead of $\otimes_\fld$.)
These also form products and co-products in the subcategory $\ts_G$, if all
$A_i\in\ts$. For an object $A\in\ts$ and $\ell\in\mathbb{N}$
we define
$$A^\ell:=\prod_{i=1}^\ell A:= A\times A\times\cdots\times A\ {\rm and}\
A^{\otimes\ell}:=\coprod_{i=1}^\ell A:= A\otimes A\otimes\cdots\otimes A$$
with $\ell$ copies of $A$ involved. This allows for the following partial characterization
of \emph{generators} in $\ts$:

\begin{lemma}\label{generators in Ts}
An object $\Gamma\in\ts$ is a generator if for every $R\in\ts$ there is a surjective
morphism $\Psi:\ \Gamma^{\otimes\ell}\to R$ for some $\ell\ge 1$.
\end{lemma}
\begin{proof} By assumption we have the following commutative diagram
\begin{diagram}\label{Delta_diagram1}
\Gamma^{\otimes\ell} &\rOnto^\Psi          &R \\
 \uTo^{\tau_i}     &\ruTo_{\Psi_i}       & \\
\Gamma             &                     & \\
\end{diagram}
where $\Psi$ is surjective. Let $\alpha,\beta\in\ts(R,S)$ with $\alpha\circ\psi=\beta\circ\psi$ for
all $\psi\in\ts(\Gamma,R)$. Then $\alpha\circ\Psi\circ\tau_i=\beta\circ\Psi\circ\tau_i$ for
all $i$, hence $\alpha\circ\Psi=\beta\circ\Psi$. Since $\Psi$ is surjective
it follows that $\alpha=\beta$, so $\Gamma$ is a generator in $\ts$.
\end{proof}

Assume that $p$ divides $|G|$, then any $A\in\ts$ is generated by finitely many points, say, $a_1,\cdots,a_\ell$. Hence
there is a surjective morphism $D_{reg}^{\otimes\ell}\to A$, mapping
$t_i\mapsto a_i$, where
$t_i:=1\otimes\cdots 1\otimes x_e\otimes 1\cdots\otimes 1$ with $x_e$ in the $i^{th}$ tensor factor.
Hence

\begin{lemma}\label{D_reg_sproj_gen}
If $p$ divides $|G|$, then $D_{reg}$ is an s-projective s-generator in $\ts$.
\end{lemma}

The remaining results in this section are straightforward generalizations
of corresponding results in the case where $G$ is a $p$-group (see \cite{catgal}).

\begin{lemma}\label{proj_gen_univ}
Let $p$ be a divisor of $|G|$ and $A\in\ts$. Then the following hold:
\begin{enumerate}
\item The algebra $A$ is s-projective if and only if $A$ is a retract of some tensor power $D_{reg}^{\otimes N}$.
\item If $A$ is s-projective or an $s$-generator, then $A$ is universal.
\end{enumerate}
\end{lemma}
\begin{proof}
(1):\  Assume that $A$ is s-projective. Then, by \ref{D_reg_sproj_gen} there is a surjective
morphism $D_{reg}^{\otimes N}\to A$, which splits since $A$ is s-projective. It follows that
$A$ is a retract of $D_{reg}^{\otimes N}$.\\
(2):\  Clearly $A$ is universal if it is an s-generator. If $A$ is s-projective, it is
universal, as a retract of the universal algebra $D_{reg}^{\otimes N}$.
\end{proof}

\begin{lemma}\label{basics_exist}
Let $X\in\ts$ be a subalgebra of $D_{reg}$ and let
$\hat X$ denote its normal closure in ${\rm Quot}(X)$. Then
$\hat X$ is universal in $\ts$. Moreover if $X$ is a subalgebra of minimal
Krull-dimension in $D_{reg}$, then $X$ and $\hat X$ are basic domains.
\end{lemma}

\begin{proof}
The polynomial ring $D_{reg}$ is a universal domain of Krull-dimension
$|G|-1$. Let $X\hookrightarrow D_{reg}$ be an embedding in $\ts$,
then $X$ is a universal domain. Now suppose that $X$ has minimal Krull-dimension.
If $Y\prec X$, then $\Dim(Y)=\Dim(X)$, but there is $\alpha\in\ts(X,Y)$ with
$\alpha(X)\prec Y\prec X$. It follows that $\Dim(\alpha(X))=\Dim(Y)=\Dim(X)$, so
$\ker(\alpha)=0$ and $X\prec Y$. This shows that $X$ is a universal minimal, hence basic, domain.\\
Since $X$ is a finitely generated $\fld$-algebra, so is $\hat X$ and, since $D_{reg}$
is a normal ring, $\hat X\le D_{reg}$. It follows that
$\hat X$ is universal, and basic, if $X$ is.
\end{proof}

The next result describes properties of basic objects and shows
that they form a single $\thickapprox$-equivalence class consisting of integral domains,
all of which have the same Krull-dimension:

\begin{prp}\label{universal_minimal_domains}
Let $A\in\ts$ be universal. Then the following are equivalent:
\begin{enumerate}
\item $A$ is basic;
\item $A$ is a basic domain;
\item every $\alpha\in \End_\ts(A)$ is injective;
\item $A\prec B$ for every universal $B\in \ts$;
\item $A\thickapprox B$ for one (and therefore every) basic object $B\in \ts$;
\item no proper quotient of $A$ is universal;
\item no proper quotient of $A$ is a subalgebra of $A$.
\end{enumerate}
Any two basic objects are $\thickapprox$-equivalent domains of the same Krull-dimension
$d_\fld(G)\le sm$ where $|G|=p^s\cdot m$ with $\gcd(p,m)=1$.
With $\frak{B}$ we denote the $\thickapprox$-equivalence class of basic objects in $\ts$.
\end{prp}
\begin{proof}
Let $X\in\ts$ be a basic domain and $\alpha\in \End_\ts(X)$. Then $\alpha(X)\prec X$, hence
$X\prec \alpha(X)$, so $\Dim(X)=\Dim(\alpha(X))$ and $\alpha$ must be injective.\\
``(1) $\Rightarrow$ (2)":\  There is $\beta\in \ts(X,A)$ and $\gamma\in\ts(A,X)$,
so $\gamma\circ\beta\in\End_\ts(X)$ is injective, which implies that $\beta$ is injective
and therefore $X\prec A$. It follows that $A\prec X$, hence $A$ is a domain.\\
``(2) $\Rightarrow$ (3)":\  This has already been shown above.
(We didn't use the fact that $A$ is universal, there. So every minimal domain in $\ts$
satisfies (3)).\\
 ``(3) $\Rightarrow$ (4)":\  Since $A$ and $B$ are universal there exist morphisms
$\alpha\in\ts(A,B)$ and $\beta\in\ts(B,A)$ with $\beta\circ\alpha$ injective, because $A$
is minimal. Hence $A\prec B$.\\
``(4) $\Rightarrow$ (5)":\  This is clear.\\
``(5) $\Rightarrow$ (1)":\  $B\thickapprox A$ means that
$B\hookrightarrow A$ and $A\hookrightarrow B$. In that case $A$ is universal (minimal)
if and only if $B$ is universal (minimal). Choosing $B:=X$, it follows that $A$ is basic.\\
``(3)$\Rightarrow$ (6)":\  Now assume that every $\alpha\in \End_\ts(B)$ is injective and let
$B/I$ be universal for the $G$-stable ideal $I\unlhd B$. Then there is
$\gamma\in\ts(B/I,B)$ and the composition with the canonical map $c:\ B\to B/I$ gives
$\gamma\circ c\in\End_\ts(B)$. It follows that $I=0$.\\
``(6) $\Rightarrow$ (1)":\  Assume $B\prec A$. Then $B$ is universal and since $A$ is universal, there is
$\theta\in\ts(A,B)$ with $\theta(A)\le B$ universal. Hence $A\cong\theta(A)\prec B$ and
$A$ is basic. \\
Let $A,B\in\ts$ be basic, then $\ts(A,B)\ne\emptyset\ne \ts(B,A)$
implies that $A\prec B\prec A$, hence $A\thickapprox B$ and $\Dim(A)=\Dim(B)=:d_\fld(G)\le sm$,
since $\ind(U_P)$ is universal of dimension $sm$ by
Theorem \ref{arb_p_grp_sec0}.
Assume that $d_\fld(G)=0$. Then $X$ must be a Galois-field extension $K\ge\fld$ with Galois group $G$
and $K\hookrightarrow D_{reg}$, which implies $K=\fld$ and $G=1$.\\
``(6) $\Rightarrow$ (7)":\  This is clear, because a quotient $A/I$ as subalgebra of $A$ would be universal.\\
``(7) $\Rightarrow$ (1)":\ We have $X\prec A$ and there is
$\theta\in\ts(A,X)$ with $\theta(A)\le X$ universal. It follows that $\theta(A)\prec A$, hence
$\ker\theta=0$ and $\theta(A)\cong A\thickapprox X$, so $A$ is basic.
\end{proof}

\begin{cor}\label{universal_minimal_domains_cor1}
Let $A\in\ts$ be a universal domain. Then $d_\fld(G)\le\Dim(A)$ and the following
are equivalent:
\begin{enumerate}
\item $A\in\frak{B}$;
\item $d_\fld(G)=\Dim(A)$;
\item If $C\in \ts$ with $C\prec A$, then $\Dim(C)=\Dim(A)$.
\end{enumerate}
\end{cor}
\begin{proof} The first statement and ``(1) $\Rightarrow$ (2)" follow immediately
from Proposition \ref{universal_minimal_domains}.\\
``(2) $\Rightarrow$ (3)":\  $C\prec A$ implies that $C$ is a universal domain and
$\Dim(C)\le\Dim(A)$. Hence $\Dim(A)=d_\fld(G)\le\Dim(C)\le \Dim(A)$.\\
``(3) $\Rightarrow$ (1)":\  Suppose $A$ is \emph{not} minimal. Then there
is $\alpha\in \End_\ts(A)$ with $\ker(\alpha)\ne 0$. Hence
$A/\ker(\alpha)\cong\alpha(A)=:C\prec A$. Clearly $\Dim(C)<\Dim(A)$.
\end{proof}

\section{Induction, co-induction and restriction}

From now on $H\le G$ will denote a subgroup of index $m:=[G:H]$
and $\R:=\R_{H/G}\subseteq G$ will be a fixed cross-section of right $H$-cosets.
Consider the \emph{Frobenius-embedding}
$$\rho_\R:\ G\hookrightarrow \hat G:=H\wr\Sigma_\R=H^\R{\rtimes}\Sigma_\R,\ g\mapsto (\bar g,\dot g),$$
where the permutation $\dot g$ and the function $\bar g\in H^\R$
are defined by the equation $rg=\bar g(r)\cdot r^{\dot g}$. Let
$\R'$ be a different cross-section of right $H$-cosets, then
$\R'=\{r':=\bar h(r)\cdot r\ |\ r\in\R\}$ with some function $\bar h\in H^\R$.
Then $Hr'g=Hrg=Hr^{\dot g}=Hr'(r^{\dot g})$, so
$\rho_{\R'}(g)=(\overline{\bar g},\dot g)$ with ``new" function
$\overline{\bar g}\in H^\R$, but the same permutation $\dot g\in\Sigma_\R$. From the equation $r'(r)g=\bar h(r)rg=\bar h(r)\bar g(r)r^{\dot g}=\overline{\bar g}(r)r'^{\dot g}=$
$\overline{\bar g}(r)\bar h(r^{\dot g})r^{\dot g}$
we conclude that
$\overline{\bar g}(r)=\bar h(r)\overline{g}(r)\bar h(r^{\dot g})^{-1}$, hence
$\rho_{\R'}=(\bar h,\id_\R)\cdot (\overline g,\dot g)\cdot (\bar h,\id_\R)^{-1}$.
This shows that $\rho_{\R'}={\rm int}(\bar h,\id_\R)^{-1}\circ \rho_\R$,
where ${\rm int}(\bar h,\id_\R)$ denotes the inner automorphism of $H\wr\Sigma_\R$ given
by right conjugation with the element $(\bar h,\id_\R)\in H\wr\Sigma_\R$.
Let $X$ be any group, then every group-homomorphism
$\theta:\ H\to X$ induces a canonical group homomorphism
$$\theta_\R:\ H\wr\Sigma_\R\to X\wr\Sigma_\R,\ (\bar h,\sigma)\mapsto (\theta\circ\bar h,\sigma).$$
If $\Omega$ is a right $H$-set via a homomorphism $\omega:\ H\to\Sigma_\Omega$, then
$\omega_\R$ induces a right $\hat G$-action on the set $\Omega^\R$ of functions from $\R$ to $\Omega$, given by
the formula
$$\phi^{(f,\sigma)}(r)=(\phi^f)(r^{\sigma^{-1}})=(\phi(r^{\sigma^{-1}}))\cdot f(r^{\sigma^{-1}}).$$
Note that for $g\in G$ we then have
$$\phi^g(r)=
(\phi(r^{{\dot g}^{-1}}))\cdot \bar g(r^{{\dot g}^{-1}})$$ with
$\bar g(r^{{\dot g}^{-1}})={r^{{\dot g}^{-1}}\cdot g\cdot r^{-1}}\in H$.
Let $\Omega=V\in mod-\fld H$ with corresponding homomorphism
$\omega:\ H\to {\rm GL}(V)$ and $\omega_\R:\ \hat G\to {\rm GL}(V)\wr\Sigma_\R$.
Then $V^\R\in \fld G-Mod$ and the correspondence
$$(\alpha:\ V\to W\in \fld H-Mod) \mapsto
(\alpha^\R:\ V^\R\to W^\R\in \fld G-Mod),$$
with $\alpha^\R(\phi)(r):=\alpha(\phi(r))$ for all $r\in\R$, is a functor.
Since $\rho_{\R'}={\rm int}(\bar h,\id_\R)^{-1}\circ \rho_\R$, we see that
different choices of cross-sections may result in different, but isomorphic
functors, where the isomorphism is given by conjugation with elements from
the base group $H^\R\le H\wr\Sigma_\R$. For simplicity we will always choose
cross-sections $\R$ in such a way that $1_H=r_1\in \R$.
Note that $_H\fld G\cong\oplus_{r\in\R}\fld H\cdot r\cong{_H}(\fld H)^\R\in \fld H-mod$,
hence $V^\R\cong (V\otimes_{\fld H}\fld H)^\R\cong$
$V\otimes_{\fld H}(\fld H)^\R\cong V\otimes_{\fld H}\fld G$,
which shows that $V^\R\cong \ind(V)$, the well known
\emph{(Frobenius-)induction of modules}.
For every $W\in mod-\fld G$, the $H-G$-bimodule structure on $\fld G$ turns
${\rm Hom}_{\fld G}(\fld G_G,W)$ into an $H$-right module with
$fh(x):=f(hx)$ for $h\in H$, $f\in{\rm Hom}_{\fld G}(\fld G_G,W)$ and
$x\in\fld G$. The map $f\mapsto f(1_G)$ is then an isomorphism
${\rm Hom}_{\fld G}(\fld G_G,W)\cong W_{|H}$ as right $H$-modules, and
it follows from the adjointness theorem for tensor- and Hom-functors, that the induction
functor $\ind()$ is left adjoint to the restriction functor from $mod-\fld G$ to
$mod-\fld H$. Due to the fact that $\fld G$ is a symmetric algebra (in the sense of
the theory of Frobenius-algebras), $\ind()$ is also right left adjoint to the
restriction functor from $mod-\fld G$. This is the content of classical
Frobenius-reciprocity and Nakayama-relations in representation theory of finite groups.
Now we consider the analogue of these in the theory of $G$-representations in the category
$\fld{\tt alg}$ of commutative $\fld$-algebras.

{\sl From now on the terms ``$\fld$-algebra" and ``$\fld-G$-algebra" will
always mean {\bf commutative} $\fld$-algebra or $\fld-G$-algebra.}

\noindent
Let $V^{\otimes\R}:=V^{\otimes m}=V^{(1)}\otimes V^{(2)}\otimes\cdots\otimes V^{(m)}$
denote the tensor product of $m=|\R|$ copies of $V$. Then there is a canonical ``tensor map" (\emph{not} a homomorphism)
$$t:\ V^\R\mapsto V^{\otimes\R},\ \phi\mapsto t_\phi:=\phi(r_1)\otimes\cdots\otimes\phi(r_m),$$
and a natural action of ${\rm GL}(V)\wr\Sigma_\R$ and $\hat G$ on
$V^{\otimes\R}$, defined by the effect on elementary tensors, following the rule
$(t_\phi)^\gamma:=t_{\phi^\gamma}$ for $\phi\in V^\R$ and $\gamma\in\hat G$.
The restriction to $G$,
$$\tensind(V):=V^{\otimes G}:=
(V^{\otimes\R})_{|G}\in mod-\fld G$$
is called the \emph{tensor-induction} of $V$. As with ordinary induction above,
it is well known and easy to see
that $\tensind()$ defines a functor from $mod-\fld H$ to $mod-\fld G$.
Again different choices for $\R$ yield isomorphic functors, twisted by
conjugation with elements from $H^\R$.
If $V=A\in\Halg$ is a commutative $\fld-H$-algebra, then
$\ind(A)=A^{\times\R}$ is a commutative $\fld-G$-algebra with ``diagonal
multiplication", such that ${\rm Res}_{|1_G}(\ind(A))=\prod_{r\in\R}A^{(r)}\in\fld{\tt alg}$.
Similarly $A^{\otimes G}$ is a commutative $\fld-G$-algebra with ``tensor multiplication"
$$(\otimes_{r\in\R}a_r)\cdot (\otimes_{r\in\R}a'_r)=
\otimes_{r\in\R}a_r a'_r,$$
such that ${\rm Res}_{|1}(\tensind(A))=\coprod_{r\in\R}A^{(r)}$,
and both are functors from $\Halg$ to $\Galg$.
For every $a\in A$ let $\hat a\in A^\R$ denote the function
with $\hat a(1_H)=a$ and $\hat a(r)=1_A$ for every $r\ne 1_H$, hence
$t_{\hat a}=a\otimes 1\otimes\cdots\otimes 1$. Then
${\hat a}^g(r)=\begin{cases}a\cdot g\cdot r^{-1}&\text{if}\ rg^{-1}\in H\\
1_A&\text{otherwise}
\end{cases}$,
so $(t_{\hat a})^g=1\otimes\cdots\otimes a\cdot g (r_1^{\dot g})^{-1}\otimes 1\cdots\otimes 1\in A^{\otimes\R}$, with non-trivial entry in position $r_1^{\dot g}$. For every
$\phi\in A^\R$ we have
$$t_\phi=\phi(1)\otimes\phi(r_2)\otimes\cdots\otimes\phi(r_m)=
t_{\widehat{\phi(1_H)}}\cdot (t_{\widehat{\phi(r_2)}})r_2\cdots (t_{\widehat{\phi(r_m)}})r_m$$
and we see that $A^{\otimes\R}=A[(t_{\hat a})r\ |\ r\in\R]$.

\begin{lemma}\label{G-hom_ident}
Suppose $A$ is a $\fld-H$-algebra, $B$ a $\fld-G$-algebra and
$\beta$ and $\gamma$ are $G$-equivariant algebra homomorphisms from $A^{\otimes\R}$ to $B$.
Then $\beta=\gamma$ if and only if $\beta(t_{\hat a})=\gamma(t_{\hat a})$ for all $a\in A$.
\end{lemma}
\begin{proof} ``Only if" is clear, so assume
$\beta(t_{\hat a})=\gamma(t_{\hat a})$ for all $a\in A$.
Then we have for every $\phi\in A^\R$:
$$\beta(t_\phi)=\prod_{i=1}^m \beta[(t_{\widehat{\phi(r_i)}})r_i]=
\prod_{i=1}^m [\beta(t_{\widehat{\phi(r_i)}})]r_i=
\prod_{i=1}^m [\gamma(t_{\widehat{\phi(r_i)}})]r_i=\gamma(t_\phi),$$
hence $\beta=\gamma$.
\end{proof}

As mentioned above, the next result provides an analogue of
Frobenius-reciprocity and Nakayama-relations in representation theory
of finite groups. Notice that, unlike in the category of modules over group algebras,
the restriction functor now has different left and right adjoints:

\begin{thm}\label{res_adjoints}
Let $A$ be a $\fld-H$-algebra and $B$ a $\fld-G$-algebra. Let
$\iota_1:\ A\hookrightarrow A^{\otimes\R}$ be the embedding
$a\mapsto t_{\hat a}$. Then the following hold
\begin{enumerate}
\item The map
$\chi:\ \Halg(A,B_{\downarrow H})\to
\Galg(\tensind(A),B ),\ \alpha\mapsto\otimes_{r\in\R}\alpha()\cdot r$
is a bijection with inverse given by $\iota_1^*:\ \beta\mapsto \beta\circ\iota_1$.
\item The map
$\psi:\ \Halg(B_{\downarrow H},A)\to
\Galg(B,\ind(A)),\ \alpha\mapsto\times_{r\in\R}\alpha(()\cdot r^{-1})$
is a bijection with inverse given by $\pi_{1*}:\ \beta\mapsto(b\mapsto \beta(b)(1_H))$.
\item The tensor induction functor $\tensind:\ \Halg\to \Galg$
is left-adjoint and the Frobenius induction functor $\ind:\ \Halg\to \Galg$
is a right-adjoint to the restriction functor $\res:\ \Galg\to
\Halg$.
\end{enumerate}
\end{thm}
\begin{proof}
(1):\ Let $\alpha\in\Halg(A,B_{\downarrow H})$. For every $r\in\R$ the map
$\alpha()\cdot r$ is a $\fld$-algebra homomorphism
from $A$ to $B$. Since $A^{\otimes\R}=\coprod_{r\in\R} A^{(r)}$ is the coproduct
in $\fld{\tt alg}$, it follows that $\chi(\alpha)\in\fld{\tt alg}(A^{\otimes\R},B)$. We now show that $\chi$ maps
$H$-morphisms to $G$-morphisms:\\
For every $\phi\in A^\R$ and $g\in G$ we have
$$[\otimes_{r\in\R}\alpha(t_\phi)\cdot r]g=[\prod_{r\in\R}\alpha(\phi(r))\cdot r]g=
\prod_{r\in\R}[\alpha(\phi(r))\cdot rg]=
\prod_{r\in\R}[\alpha(\phi(r^{{\dot g}^{-1}}))\cdot r^{{\dot g}^{-1}}g]=$$
$$\prod_{r\in\R}[\alpha(\phi(r^{{\dot g}^{-1}}))\cdot r^{{\dot g}^{-1}}gr^{-1}]r=_{(\text{since}\
r^{{\dot g}^{-1}}gr^{-1}\in H)}\
\prod_{r\in\R} [\alpha(\phi(r^{{\dot g}^{-1}})\cdot r^{{\dot g}^{-1}}gr^{-1})]\cdot r=$$
$\prod_{r\in\R}\alpha(\phi^g(r))\cdot r=\otimes_{r\in\R}\alpha(t_{\phi^g})\cdot r.$
This shows that $\chi(\alpha)((t_\phi)g)=[\chi(\alpha)(t_\phi)]g$, hence
$\chi(\alpha)$ is $G$-equivariant.
Clearly $\iota_1^*$ maps $G$-morphisms to $H$-morphisms and
$\iota_1^*\circ\chi(\alpha)=\alpha$ for every $\alpha$. For $\beta\in
\Galg(\tensind(A),B)$ we have
$\chi\circ\iota_1^*(\beta)(t_{\hat a})=$
$$\beta\circ\iota_1(a)\cdot 1\otimes\beta\circ\iota_1(1_A)\cdot r_2\otimes\cdots\otimes
\beta\circ\iota_1(1_A)\cdot r_m=
\beta\circ\iota_1(a)\otimes1_A\otimes\cdots\otimes1_A=\beta(t_{\hat a}).$$
Hence $\chi\circ\iota_1^*(\beta)=\beta$ by Lemma \ref{G-hom_ident}.\\
(2):\ This follows from the ``usual" adjointness of
the functor pair $(\res, \ind)$ in representation theory,
together with the description of the product in $\fld{\tt alg}$.
It is straightforward to confirm that the given maps are well-defined and mutually inverse
algebra morphisms.\\
(3):\ This follows immediately from (1) and (2).
\end{proof}

\begin{rem}\label{comp_with_G-SET}
\begin{enumerate}
\item Theorem  \ref{res_adjoints} has an analogue in the theory
of permutation sets. Here the functor $\Omega\mapsto\Omega^\R$ from $H$-sets to
$G$-sets is the analogue of ``tensor-induction", however, it turns out to be a
\emph{right adjoint} of the restriction functor, whereas a left adjoint is
given by the functor which maps $\Omega\mapsto\Omega\times G/H$, the $G$-set
of $H$-orbits on the cartesian product $\Omega\times G$ with $H\times G$-action
$(\omega\times g)(h,g'):=\omega\cdot h\times h^{-1}gg').$
\item Let $L:=\tensind$ and $F:=\res$ and consider the unit and co-unit
$u:=u^{(L,F)}$, $c:=c^{(L,F)}$ as in Theorem \ref{adj_form}. Then
$$u_A:\ A\to A\otimes 1\otimes\cdots\otimes 1\le \res(A^{\otimes \R})\in\Halg,\ a\mapsto t_{\hat a}$$ is the canonical embedding and
$$c_B:\ (B_{|H})^{\otimes\R}\to B,\ {\rm maps}\ b^{(1)}\otimes b^{(r_2)}\otimes\cdots\otimes b^{(r_m)}\mapsto \prod_{r\in\R}b^{(r)}\cdot r.$$
If $L=\res$ and $F=\ind$, then
$u_B:\ B\to (B_{|H})^{\times\R}$ maps $b\mapsto (b,br_2^{-1},\cdots,br_m^{-1})$,
whereas
$c_A:\ (A^{\times\R})_{|H}\to A$ maps $\bar a\mapsto \bar a(1_H)$.
\item It is well known that right adjoint functors are strongly left continuous (i.e. it they
respect limits) and $L$ is strongly right continuous (i.e. it respects colimits).
Hence $\ind$ respects limits (e.g. products, kernels, monomorphisms and injective maps) and
$\tensind$ respects colimits (e.g. coproducts, cokernels, epimorphisms and surjective maps).
\end{enumerate}
\end{rem}

Consider the unit $u:=u^{(\tensind,\res)}:\ A\hookrightarrow \res(\tensind(A))$
and the co-unit $c':=c^{(\res,\ind)}:\ (A^{\times\R})_{|H}\to A$.
The ``multiplication map" $\mu:\ \res(\tensind(A))\to A,\ \mu(t_f)\mapsto\prod_{r\in\R}f(r)$
and the ``constant map" ${\rm const}:\ A\hookrightarrow \res(\ind(A)),\ a\mapsto (a,a,\cdots,a)$
satisfy $\mu\circ u=\id_A=c'\circ {\rm const}$, hence they respectively split
$u$ and $c'$ in $\alg$, but not necessarily in $\Halg$.

\begin{lemma}\label{p_nilpotent}
Assume that $\R\subseteq G$ is normalized by $H$ (e.g. $\R\unlhd G$ is a normal subgroup
with complement $H$). Then the maps
$\mu$ and ${\rm const}$ are $H$-equivariant and therefore they
split the unit $u$ and the co-unit $c'$ in $\Halg$, respectively.
\end{lemma}
\begin{proof}
The hypothesis implies $rh=\bar h(r)r^{\dot h}=hh^{-1}rh=hr^h$,
hence $r^{\dot h}=r^h$ and $\bar h(r)=h$ for all $r\in\R$.
Therefore
$\mu(t_fh)=\prod_{r\in\R}(f^h)(r)=$ $\prod_{r\in\R}(f(r^{{\dot h}^{-1}}))\cdot \bar h(r)=$
$\prod_{r\in\R}(f(r^h))\cdot h=\mu(t_f)\cdot h$. Similarly
$[({\rm const}(a))\cdot h](r)=[({\rm const}(a))](r^{\dot h})\cdot\bar h(r)=
{\rm const}(ah)(r)$. Hence $\mu$ and ${\rm const}$ are in $\Halg$. The rest follows
from the previous remarks.
\end{proof}

In general we have the following result about the splitting behaviour of
$u$ and $c'$:

\begin{cor}\label{retract of res}
Let $A\in\Halg$, then the following are equivalent:
\begin{enumerate}
\item $A$ is a retract of $\res(B)$ for some $B\in\Galg$;
\item $A$ is a retract of $\res(\tensind)(A)$;
\item the unit $u^{(\tensind,\res)}_A$ splits in $\Halg$;
\item $A$ is a retract of $\res(\ind)(A)$;
\item the counit $c^{(\res,\ind)}_A$ splits in $\Halg$.
\end{enumerate}
\end{cor}
\begin{proof}
This follows immediately from Proposition \ref{image_of_T} of
Appendix \ref{adj_funct}
\end{proof}

\section{Frobenius reciprocity in $\ts_G$}

Let $X\le Y\le G$ be subgroups and $W\in\fld Y-mod$, then $W^X$ denotes
the subspace of $X$-fixed points and we denote with $t_X^Y$ the \emph{relative trace} map
$$t_X^Y:\ W^X\to W^Y,\ w\mapsto \sum_{g\in \R_{X\backslash Y}} wg\in W^Y,$$
where $\R_{X\backslash Y}$ denotes a cross-section of right $X$-cosets in $Y$, satisfying
$Y=\cup_{g\in \R_{X\backslash Y}} Xg$. It is easy to see that $t_X^Y$ is a linear transformation, which is
independent of the choice of the cross-section $\R_{X\backslash Y}$.
Moreover, $\tr_Y(w)=t_1^Y(w)=t_X^Y\circ t_1^X(w)$ and
since $Y=X\cdot \R_{X\backslash Y}=Y^{-1}=\R_{X\backslash Y}^{-1}\cdot X$ we have
$$\tr_Y(w)=\sum_{y\in \R_{X\backslash Y}}\sum_{x\in X}wxy=
\sum_{x\in X}\sum_{y\in \R_{X\backslash Y}}wy^{-1}x=\tr_X(w')$$
with $w':=\sum_{y\in \R_{X\backslash Y}}wy^{-1}$.

\begin{lemma}\label{ts_restrict}
Assume $m:=[G:H]$ is coprime to $p=\chr(\fld)$ and let $A\in\Galg$.
Then $A\in\ts_G\iff \res(A)\in\ts_H.$
\end{lemma}
\begin{proof}
If $a\in A_{|H}$ is a point, then $t_1^G(\frac{a}{m})=$
$t_H^G(t_1^H(\frac{a}{m}))=$
$t_H^G(\frac{1}{m})=1,$ so $\frac{a}{m}$ is a $G$-point in $A$.
If $a'\in A$ is a $G$-point, then
$1=t_1^G(a')=t_1^H(a'')$ with $a''=\sum_{g\in \R_{H\backslash G}}a'g^{-1}$,
hence $a''\in A$ is an $H$-point.
\end{proof}

\begin{prp}\label{tens_ind_ts}
Assume that $p$ does not divide $m=[G:H]$, then the following hold:
\begin{enumerate}
\item If $A\in\ts_H$, then $\tensind(A)$ and $\ind(A)\in\ts_G$ and if $B\in\ts_G$, then
$\res(B)\in\ts_H$. In particular restriction to $\ts_H$ and $\ts_G$
induces the adjoint pairs of functors
$$(\tensind_{|\ts_H},\res_{|{\ts_G}})\ {\rm and}\
(\res_{|{\ts_G}},\ind_{|{\ts_H}}).$$
\item If $A\in\Halg$, then $A\in\ts_H\iff\ind(A)\in\ts_G$.
\end{enumerate}
\end{prp}
\begin{proof}
(1):\ Let $a\in A\in\ts_H$ and set $\hat a:=(a,1,\cdots,1)\in A^\R$.
Then $\tr_G(\hat a)=\tr_H^G(\tr_H(\hat a))=\tr_H^G(\widehat{\tr_H(a)})$ and
$\tr_G(t_{\hat a})=\tr_H^G(\tr_H(t_{\hat a}))=\tr_H^G(t_{\widehat{\tr_H(a)}})$.
It follows that if $a$ is an $H$-point in $A$, then
$1/m\cdot\hat a$ is a ($G$-) point in $\ind(A)$ and
$1/m\cdot t_{\hat a}$ is a ($G$-) point in $\tensind(A)$. The rest of the statement follows from
Lemma \ref{ts_restrict} and Theorem \ref{res_adjoints}.\\
(2):\ It suffices to show ``$\Leftarrow$". Let $A\in\Halg$ with $\ind(A)\in\ts_G$; the
counit $c:\ \res(\ind(A))=(A^{\times\R})_{|H}\to A$, mapping
$\bar f\mapsto \bar f(1_H)$ is $H$-equivariant (and surjective), so its maps $H$-points to
$H$-points. By Lemma \ref{ts_restrict}, $\res(\ind(A))$ has $H$-points, hence so does
$A$ and therefore $A\in\ts_H$.
\end{proof}

\begin{prp}\label{preserved_props}
Assume that $p$ does not divide $m=[G:H]$.
If $A\in\ts_H$ is (a) s-projective, (b) an s-generator, (c) universal or (d) cyclic,
then so is $\tensind(A)\in\ts_G$.
\end{prp}
\begin{proof}
(a):\ This follows from Proposition \ref{s-prj_cat} (1).\\
(b):\ Suppose that $A\in\ts_H$ is an $s$-generator and $B\in\ts_G$. Then
there is a surjective morphisms $\beta:\ A^{\otimes N}=\coprod_{i=1}^N A\to \res(B)$.
Hence $\tensind(\beta):\ \tensind(A^{\otimes N})\cong \tensind(A)^{\otimes N}\to\tensind(\res(B))$
is surjective by Remark \ref{comp_with_G-SET} (2). Since the counit $c_B:\ \tensind(\res(B))\to B$
of Remark \ref{comp_with_G-SET} (3) is surjective, it follows that $\tensind(A)$ is an s-generator.\\
(c):\ Suppose that $A\in\ts_H$ is universal and $B\in\ts_G$. Then there exists
$\alpha\in\ts_H(A,\res(B))$, hence $\chi(\alpha)\in\ts_G(\tensind(A),B)$, so $\tensind(A)$ is
universal in $\ts_G$.\\
(d):\ If $A\in\ts_H$ is cyclic, then $A=\fld[(a)h\ |\ h\in H]$ for some point $a\in A$.
It follows by construction that
$\tensind(A)=$ $\fld[t_{\hat a}hr\ |\ h\in H, r\in\R]=$ $\fld[t_{\hat a}g\ |\ g\in G]$,
hence $\tensind(A)$ is cyclic.
\end{proof}

\begin{lemma}\label{prj_gen_then_univ}
Let $B\in\ts^o_G$. If $B$ is s-projective or an s-generator, then $B$ is universal.
\end{lemma}
\begin{proof}
Since $B$ is generated by a finite set of points, there is a surjective
morphism $D_{reg}(G)^{\otimes N}\to B$. If $B$ is s-projective, this map splits
and $B$ is a retract of the universal algebra $D_{reg}(G)^{\otimes N}$. Hence $B$ is universal.
Clearly if $B$ is an s-generator, then $B$ is universal.
\end{proof}

\begin{cor}\label{prj_retract of DregN}
$B\in\ts^o_G$ is s-projective if and only if $B$ is a retract of some
$D_{reg}(G)^{\otimes N}$.
\end{cor}

\begin{lemma}\label{standard_equiv}
Let $B\in\ts_G$ and $A\in\ts_H$ with $m=[G:H]$ coprime to $p={\rm char}(\fld)$.
Then the following hold:
\begin{enumerate}
\item $B$ is standard $\iff$ $B$ is cyclic and s-projective.
\item $A$ is standard $\Rightarrow$ $\tensind(A)$ is standard.
\end{enumerate}
\end{lemma}
\begin{proof}
(1):\ ``$\Rightarrow$" is clear. ``$\Leftarrow$":\  Let $B$ be
cyclic and s-projective. Then $B=\fld[bg\ |\ g\in G]$ with point $b\in B$.
Hence there is a surjective morphism $D_{reg}(G)\to B$, which splits since
$B$ is s-projective. Therefore $B$ is standard.\\
(2):\ ``$\Rightarrow$": If $A$ is standard, then $A$ is cyclic and s-projective,
hence so is $\tensind(A)$ by Proposition \ref{preserved_props}. By (1), $\tensind(A)$
is standard.
\end{proof}

Let $V\in\fld G-mod$ and assume that $p\not |\ [G:H]$. Then $V$ is a direct summand of
$\ind\res(V)$. Since induced modules of projective $H$-modules are projective $G$-modules,
it follows that $V$ is projective in $\fld G-mod$ if and only if $\res(V)$ is projective in
$\fld H$-mod. The categorical ``reason" for this phenomenon is the fact that
the map
$$V\to \ind\res(V)=V_{|H}\otimes_{\fld H}\fld G,\ v\mapsto \frac{1}{m}\sum_{r\in\R} vr^{-1}\otimes_{\fld H} r\in Mod-\fld G$$
is a right inverse to the counit map $c^{(\ind,\res)}:\ \ind\res(V)\to V,\ v\otimes_{\fld H} r \mapsto vr$.
This motivates the following

\begin{df}\label{H_split}
An algebra $B\in\Galg$ or $\ts_G$ will be called $H$-split, if
the co-unit $c^{(\tensind,\res)}_B:\ \tensind(\res(B))\to B$
has a right inverse.
\end{df}

\begin{prp}\label{SV_H_split}
Let $V\in mod-\fld G$, $v\in V^G$ and assume that
$p={\rm char}(\fld)\not |\ m=[G:H]$. Consider the algebras
$S(V):={\rm Sym}(V)$ and $\overline{S(V)}_v:=S(V)/(v-1)\in\Galg$.
Then $S(V)$ and $\overline{S(V)}_v$ are $H$-split.
\end{prp}
\begin{proof}
It follows from the definition of module induction and the universal property
of ${\rm Sym}(V)$ that $\tensind(S(V))=S(\ind(V))$. It is well known and can easily
be seen that $V$ is a direct summand of $\ind(\res(V))$ by the maps
$$\theta:\ V\to \ind(\res(V)),\ w\mapsto \frac{1}{m}\sum_{r\in\R} vr^{-1}\otimes r,\ {\rm and}$$
$\mu:\ \ind(\res(V))\to V,\ v\otimes r\mapsto vr$.
Let $S(\theta):\ S(V)\to S(\ind(\res(V)))\cong \tensind(S(\res(V)))$ be the induced map in
$\Galg$ and $c_{S(V)}:\ \tensind{S(V)}\to S(V)$ the counit. Then for every $w\in V$,
$S(\theta)(w)=\frac{1}{m}\sum_{r\in\R} t_{\widehat{wr^{-1}}}\cdot r$, hence
$c_{S(V)}\circ S(\theta)(w)=\frac{1}{m}\sum_{r\in\R} c_{S(V)}(t_{\widehat{wr^{-1}}}\cdot r)=$
$\frac{1}{m}\sum_{r\in\R} c_{S(V)}(t_{\widehat{wr^{-1}}})\cdot r=$
$\frac{1}{m}\sum_{r\in\R} (wr^{-1})\cdot r=w.$
Hence $c_{S(V)}\circ S(\theta)=\id_{S(V)}$.
Now let $can:\ S(V)\to \overline{S(V)}_v,\ x\mapsto \bar x$ be the canonical map, sending
$v\mapsto 1\in\overline{S(V)}_v$.
Then we obtain the commutative diagram:
\begin{diagram}\label{SV_split_diagram}
\tensind(S(V))&\rTo^{can^\otimes}&\tensind(\overline{S(V)}_v)\\
\uTo^{S(\theta)}\dTo_{c_{S(V)}}& &\uDashto^{\overline{S(\theta)}}\dTo_{c_{\overline{S(V)}_v}}\\
S(V)&\rTo^{can}&\overline{S(V)}_v\\
\end{diagram}
Indeed, as before we see $can^\otimes\circ S(\theta)(v)=$
$\frac{1}{m}\sum_{r\in\R} can^\otimes(t_{\widehat{vr^{-1}}}\cdot r)=$
$\frac{1}{m}\sum_{r\in\R} can^\otimes(t_{\widehat{v}})\cdot r=$
$\frac{1}{m}\sum_{r\in\R} t_{\widehat{\bar v}}\cdot r=$
$\frac{1}{m}\sum_{r\in\R} 1\cdot r=1$.
This shows the existence of $\overline{S(\theta)}\in\Galg$.
Clearly $\overline{S(\theta)}$ is a right inverse of $c_{\overline{S(V)}_v}$.
\end{proof}

The first part or the next result generalizes Lemma \ref{D_is_s_proj}:
\begin{thm}\label{SV_sproj}
Let $p={\rm char}(\fld)$, $1\ne H\in{\rm Syl}_p(G)$ and $V\in mod-\fld G$.
Then the following hold:
\begin{enumerate}
\item Let $v\in V^G$ with $v^{p^N}=\tr_G(f)$ or $v^{p^N}=\tr_H(f)$ for some $f\in S(V)$ and $N\in\mathbb{N}$.
Then $S(V)/(v-1)\in\ts_G$ is s-projective.
\item If $B\in\ts_G$, then $B\otimes (S(V)/(v-1))\cong B\otimes S(V/\fld v)$ as $B$-algebras.
\item If $B\in\ts_G$ is s-projective in $\ts_G$, then so are $B\otimes S(V)$ and $B\otimes (S(V)/(v-1))$.
\end{enumerate}
\end{thm}
\begin{proof}
(1):\ Let $S:=S(V)/(v-1)$. If $v^{p^N}=\tr_G(f)$, then $v^{p^N}=\tr_H(f')$ for a suitable $f'\in S(V)$,
so we can assume $v^{p^N}=\tr_H(f)$ with $v\in V^G$. It follows from \cite{universal} Theorem 2.8 that $\res(S)$
is s-projective in $\ts_H$. Therefore $S\in\ts_G$ by Lemma \ref{ts_restrict} and
$S$ is s-projective in $\ts_G$ by Proposition \ref{s-prj_cat}, because
$S$ is $H$-split by Proposition \ref{SV_H_split}. Note that $\res$ clearly respects surjective maps.\\
(2):\ It follows from Bass' theorem that $B$ is an injective $\fld G$-module. Hence
the embedding $\fld v\hookrightarrow B,\ \lambda v\mapsto\lambda\cdot 1_B$ extends firstly to
$\phi\in\Hom_{\fld G}(V,B)$ and then to a map in $\Hom_{\fld G}(V,B\otimes S(V))$, sending  $u\in V$ to
$u-\phi(u)$, hence $v$ to $v-1$. The induced algebra-morphisms $S(V/\fld v)\to B\otimes S(V)/(v-1)$ and
$B\hookrightarrow  B\otimes S(V)/(v-1)$ induce a coproduct morphism
$\tilde\phi:\ B\otimes S(V/\fld v)\to B\otimes S(V)/(v-1)$, which is surjective in $\ts_G$ with
$\tilde\phi_{|B}=\id_B$. Since $S(V/\fld v)\cong\fld[X_1,\cdots,X_\ell]\cong S(V)/(v-1)$ with $\ell=\dim_\fld(V)-1$,
the algebras $B\otimes S(V)/(v-1)$ and $B\otimes S(V/\fld v)$ are isomorphic and
there is a morphism of $B$-algebras, $\psi:\ B\otimes S(V)/(v-1)\to B\otimes S(V/\fld v)$ with
$\tilde\phi\circ\psi=\id_{B\otimes S(V)/(v-1)}$. Hence
$B\otimes S(V/\fld v)\cong X\oplus\ I$ with ideal $I=\ker(\tilde\phi)$ and
unital subring $X\cong B\otimes S(V)/(v-1)\cong B\otimes S(V/\fld v)$.
It follows that $I=0$, since the noetherian ring $B\otimes S(V/\fld v)$ cannot be
isomorphic to a proper quotient.
We conclude that $\tilde\phi$ is an isomorphism.\\
(3):\ Let $\alpha:\ A\to A'\in \ts_G$ and $\beta:\ B\otimes S(V)\to A'$ be
morphisms in $\ts$ with $\alpha$ surjective. Since $B\otimes V$ is projective in $Mod-\fld G$, there exists
$\chi\in\Hom_{\fld G}(B\otimes V,A)$ with $\alpha\circ\chi=\beta_{|B\otimes V}$; since
$B$ is s-projective there is $\theta\in\ts_G(B,A)$ with
$\alpha\theta=\beta_{|B}$. Let ${\mathcal V}:=\{v_1,\cdots,v_\ell\}$ be a $\fld$-basis of $V$, then
we define the $\fld$-algebra morphism
$$\tilde\chi:\ B\otimes S(V)\to A,\
\sum_{\mu\in\mathbb{N}_0^\ell}b_\mu\underline v^\mu\mapsto
\sum_{\mu\in\mathbb{N}_0^\ell}\theta(b_\mu)\chi(1_B\otimes v_1)^{\mu_1}\cdots \chi(1_B\otimes v_\ell)^{\mu_\ell}.$$
Since $\alpha\circ\tilde\chi(1_B\otimes v_i)=\alpha\circ\chi(1_B\otimes v_i)=\beta(1_B\otimes v_i)$ and
$\alpha\circ\tilde\chi_{|B}=\alpha\circ\theta=\beta_{|B}$ it follows that
$\alpha\circ\tilde\chi=\beta$. Since $\tilde\chi((1_B\otimes v_i)g)=$
$\chi((1_B\otimes v_i)g)=$ $\chi(1_B\otimes v_i)g=$
$\tilde\chi(1_B\otimes v_i)g$
for all $i=1,\cdots,\ell$ and
$\tilde\chi(bg)=$ $\theta(bg)=$ $\theta(b)g=$ $\tilde\chi(b)g$,
we conclude that $\tilde\chi\in\ts_G(B\otimes S(V))$. This shows that
$B\otimes S(V)$ is s-projective. For $B\otimes (S(V)/(v-1))$ the claim follows from (2).
\end{proof}

\begin{cor}\label{split_proj_iff}
Let $p={\rm char}(\fld)\not |\ [G:H]$ and $B\in\ts_G$ be an $H$-split algebra. Then $B$ is s-projective in $\ts_G$ if and only if
$\res(B)$ is s-projective in $\ts_H$.
\end{cor}
\begin{proof} Assume that $B\in\ts_G$ is $H$-split. Then by Proposition \ref{s-prj_cat} $B$ is s-projective,
if $\res(B)\in\ts_H$ is. If $B$ is s-projective, then $B$ is a retract of $D_{reg}(G)^{\otimes N}$.
Since $D_{reg}(G)=\overline{S(V_{reg})}$ with $V_{reg}:=\fld G_{\fld G}$ and suitable $v\in V^G$, it follows from
Theorem \ref{SV_sproj} (1) that $\res(D_{reg}(G))=\overline{S(V_{reg|H})}$ is s-projective in $\ts_H$.
\end{proof}

\section{A version of Maschke's theorem}\label{Maschke-section}

\begin{lemma}\label{standard alg_decomp}
Let $\fld$ be an arbitrary field and let $A,B$ be connected
$\mathbb{N}_0$-graded $\fld$ algebras, generated in degree one (i.e. $A=\fld[A_1]$ and
$B=\fld[B_1]$). Let $\phi:\ A\to B$ be a
(not necessarily graded) homomorphism of $\fld$-algebras.
Define $\phi_i:=\pi_i\circ\phi$, where
$\pi_i:\ B\mapsto B_i$ is the projection onto the homogeneous
component of degree $i$. Then
$$\phi(A)\cap B_1\subseteq \phi_1(A)\subseteq \phi_1(A_1).$$
In particular, if $\phi$ is surjective, then
$B_1=\phi_1(A_1)$.
\end{lemma}
\begin{proof}
Let $\{x_1,x_2,\cdots,\}\subseteq A_1$ be an $\fld$-basis of $A_1$,
then $A=\fld[x_1,x_2,\cdots]$ and every $a\in A$ is an $\fld$-linear combination
of monomials $x^\alpha:=\prod_{i=1}^Nx_i^{\alpha_i}$ with
$N\in\mathbb{N}$ and $\alpha\in\mathbb{N}_0^N$.
Since $\phi(x_i)=\sum_{j\in\mathbb{N}_0} \phi_j(x_i)$, we have
$\phi(x^\alpha)=\prod_{i=1}^N(\sum_{j\in\mathbb{N}_0} \phi_j(x_i))^{\alpha_i}$.
Since $B_{\ge 2}\in\ker(\pi_1)$:
$$\pi_1\phi(x^\alpha)=\pi_1[\prod_{i=1}^N(\phi_0(x_i))^{\alpha_i}+
{\alpha_i}\phi_0(x_i)^{\alpha_i-1}\phi_1(x_i)]=$$
$$\sum_{s=1}^N \pi_1[(\prod_{i=1\atop i\ne s}^N\phi_0(x_i)^{\alpha_i})\cdot\alpha_s\cdot
\phi_0(x_s)^{\alpha_s-1}\phi_1(x_s)]\in \langle \phi_1(x_s)\ |\ s\in\mathbb{N}\rangle_\fld,$$
because $(\prod_{i=1\atop i\ne s}^N\phi_0(x_i)^{\alpha_i})\cdot\alpha_s\cdot
\phi_0(x_s)^{\alpha_s-1}\in B_0=\fld$.
\end{proof}

\begin{cor}\label{standard alg_decomp_cor}
Let $\phi:\ A\to B$ be as in \ref{standard alg_decomp}, let $G$ be a group
acting on $A$ and $B$ by graded algebra automorphisms and assume that
$\phi$ is $G$-equivariant. Then $\phi_{1|A_1}\in \Hom_{\fld G}(A_1,B_1)$.
If $\phi$ is surjective, then so is $\phi_{1|A_1}$.
\end{cor}
\begin{proof} Since the $G$-action is graded, all $A_i$ and $B_i$ are $\fld G$-subspaces
and each $\phi_i$ is $G$-equivariant. This implies the first claim.
If $\phi$ is surjective, then $B_1=B_1\cap\phi(A)\subseteq \phi_1(A_1)$, hence
$\phi_{1|A_1}$ is surjective.
\end{proof}

\begin{cor}\label{standard alg_decomp_cor2}
Let $A=\fld[A_1]$, $B=\fld[B_1]$ be graded connected $\fld$-algebras as in \ref{standard alg_decomp_cor} (with graded $G$ action) and let $V,W$ be two finite dimensional $\fld G$-modules with corresponding
symmetric $\fld$-$G$-algebras $S(V):={\rm Sym}(V)$ and $S(W):={\rm Sym}(W)$.
Then the following hold:
\begin{enumerate}
\item $S(V)\cong S(W)$ as $\fld$-$G$-algebras $\iff$ $V\cong W\in \fld G$-mod.
\item $S(V)\cong A\otimes B$ $\iff$ $V\cong A_1\oplus B_1$ with polynomial rings
$A\cong S(A_1)$ and $B\cong S(B_1)$.
\end{enumerate}
\end{cor}
\begin{proof}
(1):\ ``$\Leftarrow$" is clear. \\
``$\Rightarrow$":\ Let $\phi:\ S(V)\to S(W)$ be a $G$-equivariant algebra isomorphism.
Then $W=(S(W)_1)$ and $V=(S(V))_1$, hence by \ref{standard alg_decomp},
$W=\phi_1(V)$ with $\fld G$-homomorphism $\phi_1$. Since
$\dim_\fld(W)=\Dim(S(W))=\Dim(S(V))=\dim_\fld(V)$, it follows that
$\phi_1$ is an isomorphism.\\
(2):\ From \ref{standard alg_decomp} we obtain: $A_1\oplus B_1=(A\otimes B)_1\subseteq \phi_1(V)$.
Clearly $\Dim(A)\le \dim_\fld(A_1)$ and $\Dim(B)\le \dim_\fld(B_1)$
hence
$$\dim_\fld(V)=\Dim(S(V))=\Dim(A)+\Dim(B)\le\dim_\fld(A_1)+\dim_\fld(B_1)\le \dim_\fld(V).$$
This shows that $\Dim(A)=\dim_\fld(A_1)$ and $\Dim(B)=\dim_\fld(B_1)$, so
$A\cong S(A_1)$, $B\cong S(B_1)$ and $A\otimes B\cong S(A_1\oplus B_1)$.
From (1) we obtain $V\cong A_1\oplus B_1$.
\end{proof}

We will use the following lemma, which is well known in representation theory.

\begin{lemma}\label{tens_projective_bites}
Let $V\in \Mod-\fld G$ and $M\in\Mod-\fld H$ with $H\le G$. Then the following hold:
\begin{enumerate}
\item $\ind(M)\otimes V\cong \ind(M\otimes\res(V))\in Mod-\fld G$.
\item For $U\le V\in Mod-\fld G$ and $P\in Mod-\fld G$ projective:
$P\otimes V\cong P\otimes (U\oplus V/U)$.
\end{enumerate}
\end{lemma}
\begin{proof}
(1) the map
$\ind(M)\otimes V\to \ind(M\otimes\res(V)),\ m\otimes_Hr\otimes v\mapsto (m\otimes v r^{-1})\otimes_Hr$
is easily seen to be $G$-equivariant with the inverse
$m\otimes\ v\otimes_H r\mapsto (m\otimes_H r)\otimes vr$.\\
(2) Since $P\oplus X\cong \oplus_{i\in I} \fld G_{\fld G}^{(i)}$, we can assume that $P\cong\fld G$.
But then $\fld G\otimes V\cong {\rm Ind}_1^{\uparrow G}(\fld)\otimes V\cong$
${\rm Ind}_1^{\uparrow G}(\fld\otimes {\rm res}(V_{|1}))\cong $
${\rm Ind}_1^{\uparrow G}(\fld\otimes {\rm res}(U_{|1})\oplus {\rm res}(V/U_{|1}))\cong $
${\rm Ind}_1^{\uparrow G}(\fld\otimes {\rm res}(U_{|1}))\oplus
{\rm Ind}_1^{\uparrow G}(\fld\otimes {\rm res}(V/U_{|1}))\cong  $
$(P\otimes U)\oplus (P\otimes V/U)\cong
P\otimes (U\oplus V/U)$.
\end{proof}
Note that
(2) can be viewed as a generalized version of Maschke's theorem (taking
$P\cong\fld$ if $p\not |\ |G|$). Let $A\in\Galg$ and $\phi:\ V\to W\in Mod-\fld G$
a homomorphism of $\fld G$-modules. Then the $\fld G$-homomorphism
$V\to A\otimes S(W),\ v\mapsto 1_A\otimes \phi(v)$ extends to a morphism
$S(\phi)\in\Galg(S(V),A\otimes S(W))$. Together with the canonical embedding
$A\hookrightarrow A\otimes S(W)$, we obtain a coproduct morphism
$A\otimes\phi:\ A\otimes S(V)=A\coprod S(V)\to A\otimes S(W)$.
Hence $A\in\Galg$ induces a functor $Mod-\fld G\to \Galg,\ V\mapsto A\otimes S(V)$.
If $A\in\ts_G$, then $A\otimes ?$ is indeed a functor from $Mod-\fld G\to\ts_G$,
This reflects the well known fact from representation theory that the tensor product
of an arbitrary $\fld G$-module with a projective one is again projective.
On the other hand, $A\otimes S(V)$ can also be viewed as an $A$-algebra.
With $\GAalg$ we will denote the full subcategory of $\Galg$ consisting
of objects $B\in\Galg$ that contain $A$ as $\fld-G$ subalgebra.
Then $A\otimes ?$ is a functor from $Mod-\fld G\to \GAalg$

\begin{prp}\label{biting_off_constitutents}
For any $A\in\ts_G$ then the following hold:
\begin{enumerate}
\item If $W\le V\in Mod-\fld G$, then $A\otimes S(V)\cong A\otimes S(W)\otimes S(V/W)\in\GAalg$.
\item If $V\in mod-\fld G$, then $A\otimes S(V)\cong A\otimes S(V_{ss})\in\GAalg$, where
$V_{ss}$ is the direct sum of the simple components of $V$, including multiplicities.
\item The functor $A\otimes:\ Mod-\fld G\to\GAalg,\ V\mapsto A\otimes S(V)$ is split exact,
i.e. it maps short exact sequences to coproducts in $\GAalg$.
\item The functor $A\otimes$ induces a map from the Grothendieck group
of $Mod-\fld G$ to the set of isomorphism classes of $\GAalg$.
\end{enumerate}
\end{prp}
\begin{proof}
(1):\ Since $A\in Mod-\fld G$ is projective, we have
$A\otimes V\cong A\otimes (W\oplus V/W)$, hence there is a submodule
$X\le A\otimes V$ together with an isomorphism
$\phi:\ W\oplus V/W\cong X\in Mod-\fld G$, such that $A\otimes X=A\otimes V$.
This, together with the canonical embedding
$A\hookrightarrow A\otimes S(V)$, induces an isomorphism
$$\tilde\phi:\ A\otimes S(W)\otimes S(V/W)\cong  A\otimes S(W\oplus V/W)=\fld[A\otimes X]=
\fld[A\otimes V]\cong A\otimes S(V)\in\GAalg.$$
(2):\ This follows from an obvious induction.\\
(3) and (4):\ Note that
$A\otimes S(W)\otimes S(V/W)\cong$
$$(A\otimes S(W))\otimes_A(A\otimes S(V/W))\cong (A\otimes S(W))\coprod (A\otimes S(V/W))\in
\GAalg.$$
Now (2) and (3)  re direct consequences of (1).
\end{proof}

\begin{df}\label{A_soc}
For $A\in\Galg$ we define $A_{\rm soc}:=A_{\rm soc_{|G}}:=\fld[{\rm Soc}(A)]$, the subalgebra of $A$ generated by
the socle of the $\fld G$-module $A$.
\end{df}

Clearly if $p$ does not divide $|G|$, then $A_{\rm soc}=A$. If $0\ne W\in mod-\fld G$, then $0\ne W^{{\rm O}_p(G)}$ is $G$-stable,
so if moreover $W$ is irreducible, then ${{\rm O}_p(G)}$ acts trivially on $W$. On the other hand, if $g\in G$ acts trivially
on every irreducible $\fld G$-module, then $g-1\in {\rm Rad}(\fld G)$, the Jacobson-radical of $\fld G$ and therefore
$g^{p^n}-1=(g-1)^{p^n}=0$ for $n>>0$. From this we easily obtain the well-known formula
\begin{equation}\label{Op_kernel}
{\rm O}_p(G)=\cap_{W\in \fld G-mod\atop W\ {\rm simple}} C_G(W)
\end{equation}
with $C_G(W):=\ker_G(W)=\{g\in G\ |\ g_{|W}=\id_W\}$.

\begin{prp}\label{A_soc_prp}
Let $A\in\Galg$ with $G$ acting faithfully on $A$, then the following hold:
\begin{enumerate}
\item $A^G\le A_{\rm soc}\le A^{{\rm O}_p(G)}=A_{{\rm soc}_{|_{{\rm O}_{p,p'}(G)}}}.$
\item If $A$ is a domain, then ${\rm Quot}(A_{\rm soc})={\rm Quot}(A)_{\rm soc}=\mathbb{K}^G[{\rm Soc}(A)]$
with $\mathbb{K}:={\rm Quot}(A)$.
\item If $A$ is a normal domain, then $A^{{\rm O}_p(G)}$ is the integral closure of
$A_{\rm soc}$ in its quotient field.
\end{enumerate}
\end{prp}
\begin{proof}
(1):\ The first inequality is obvious, the second one follows from Equation (\ref{Op_kernel}) and the last
equality is clear, since $p$ does not divide $|{\rm O}_{p,p'}(G)/{\rm O}_{p}(G)|$. \\
(2):\ Let $W\le {\rm Soc}(\mathbb{K})$ be a simple $\fld G$-module.
Then there exists $0\ne a\in A^G$ such that $aW\le A_{\rm soc}$, hence $W\le {\rm Quot}(A_{\rm soc})$
and therefore
$$\mathbb{K}_{\rm soc}=\fld[{\rm Soc}(\mathbb{K})]=\mathbb{K}^G[{\rm Soc}(\mathbb{K})]=\mathbb{K}^G[{\rm Soc}(A)].$$
In particular the algebra $\mathbb{K}_{\rm soc}$ is a field containing $A_{\rm soc}=\fld[{\rm Soc}(A)]$, hence
${\rm Quot}(A_{\rm soc})\subseteq \mathbb{K}_{\rm soc}$. Since $\mathbb{K}^G\subseteq {\rm Quot}(A_{\rm soc})$ it
follows that $\mathbb{K}_{\rm soc}=\mathbb{K}^G[{\rm Soc}(A)]\subseteq {\rm Quot}(A_{\rm soc})$. \\
(3):\ Let $A$ be a normal domain. Then $\mathbb{K}^G\le\mathbb{K}_{\rm soc}$ and Galois-theory imply that
$\mathbb{K}_{\rm soc}=\mathbb{K}^X\le \mathbb{K}^{{\rm O}_p(G)}$ for some subgroup $X\le G$ containing ${\rm O}_p(G)$.
By the normal basis theorem, the $\fld G$-module $\mathbb{K}$ contains a copy of $V_{reg}$ and therefore
every simple $\fld G$-module appears in ${\rm Soc}(\mathbb{K})$. Since $X$ acts trivially on ${\rm Soc}(\mathbb{K})$,
it follows that $X\le {\rm O}_p(G)$ and therefore $X={\rm O}_p(G)$.
\end{proof}

If $G$ is a $p$-group, then clearly $A_{\rm soc}=A^G$; in this situation it has been shown
in \cite{nonlin} Proposition 4.2, that if $\theta:\ A\to B$ is a morphism in $\ts_G$, then
$B\cong B^G\otimes_{A^G}A$ and $B$ is free of rank $|G|$ over $B^G$. The next result
is a partial generalization of this:

\begin{prp}\label{A_tens_Soc}
Let $\theta:\ A\to B$ be a morphism of algebras in $\ts_G$ with $B\in\ts_G^o$. Then there is a surjective morphism
$B_{\rm soc}\otimes_{A_{\rm soc}}A\to B$, where $B$ is a right $A_{\rm soc}$-module
via the map $\theta$.
\end{prp}
\begin{proof} Since $B$ is generated by a finite set of points, there is a surjective
morphism $\phi:\ D_{reg}^{\otimes N}\otimes_\fld A\to B$. By Theorem \ref{SV_sproj},
$D_{reg}^{\otimes N}\otimes_\fld A\cong(\overline{S(V_{reg})})^{\otimes N}\otimes_\fld A\cong S(W)\otimes_\fld A$
with suitable $\fld G$-module $W$. By Proposition \ref{biting_off_constitutents},
$S(W)\otimes_\fld A\cong S(M)\otimes_\fld A$ with a suitable semisimple $\fld G$-module $M$, hence
$B=\phi(D_{reg}^{\otimes N}\otimes_\fld A)=\phi(S(M)\otimes_\fld A)={\rm Im}(B_{\rm soc}\otimes_{A_{\rm soc}} A\to B)$.
\end{proof}

\begin{cor}\label{A_tens_sym_gen}
Let $p$ be a divisor of $|G|$, let $A\in\ts_G$ be universal and $V\in mod-\fld G$ a module such that every simple $\fld G$-module is a constituent of $V$. Then $A\otimes S(V)$ is an $s$-generator in $\ts_G$.
\end{cor}
\begin{proof}
By Proposition \ref{biting_off_constitutents} we can assume that $V$ is semisimple. Let $\phi:\ A\to D_{reg}$
be a morphism in $\ts_G$,
then the proof of Proposition \ref{A_tens_Soc} shows that there is a surjective morphism
$A\otimes S(M)\to D_{reg}$ with suitable semisimple module $M$. For a suitable integer $s\ge 0$ we have
a surjective map $\theta\in \Hom_{\fld G}(V^s,M)$. Using the multiplication map $A^{\otimes s}\to A$
one can extend $\theta$ to a surjective morphism $(A\otimes S(V))^{\otimes s}\cong A^{\otimes s}\otimes S(V^s)\to A\otimes S(M)$.
Since $D_{reg}$ is an s-generator by Lemma \ref{D_reg_sproj_gen}, so are $A\otimes S(M)$ and
$A\otimes S(V)$.
\end{proof}

\begin{prp}\label{A_tens_Soc_2}
Let $G$ be $p$-solvable of $p$-length $s$ and order $hq$ with $h=p^m$ and
$p\not | q$. Let $A\in\ts_G$ with point $u\in A$ and set $C:=\fld[u^G]=\fld[(u)g\ |\ g\in G]$, then
$A\cong A_{soc}\otimes_{C_{soc}}C$ as algebra and as module over $A_{soc}$ it is generated
by elements $y_1,\cdots,y_h\in A$ of the form
$y_i=\sum_j(u)g_{i,j,1}(u)g_{i,j,2}\cdots (u)g_{i,j,s}$ for suitable $g_{i,j,k}\in G$.
\end{prp}
\begin{proof} We can assume that $G$ is neither $p$ nor $p'$-group.\\
Assume $1={\rm O}_p(G)$. Then $1\ne N:={\rm O}_{p'}(G)$ and $A^N\in \ts_G^o\cap\ts_{G/N}^o$.
Since $A^N\in\ts_G^o$  there is an onto morphism $A_{soc} \otimes_\fld A^N\to A$ in $\ts_G$, by
Proposition \ref{A_tens_Soc}. Since $1\ne {\rm O}_p(G/N)$ and $A^N\in \ts_{G/N}$, induction gives
$A^N=\sum_{i=1}^h (A^N)_{soc_{G/N}}y_i$ with $y_i$ as required. Since
$(A^N)_{soc_{G/N}}\le A_{soc}$ the result follows. \\
So we can assume that $1\ne P:={\rm O}_p(G)$. Let ${\R}_P\subseteq G$
be a cross-section of left $P$-cosets with $G=\cup_{r\in\R_P} rP$. Then $x:=\sum_{r\in\R_P}ur$ is a $P$-point
and we get $A=\oplus_{g\in P}A^P (x)g$. By induction and $A^P\in\ts_{G/P}$ we have
$A^P=\sum_{j=1}^{h/|P|} (A^P)_{soc_{G/P}} y'_j$ with $y'_j=(u')g'_{1,j}\cdots (u')g'_{s-1,j}$
and $u':=\sum_{g\in P} ug$, a $G/P$-point in $A^P$. Since $(A^P)_{soc_{G/P}}\le A_{soc_G}$,
$A=\sum_{j=1\atop g\in P}^{h/|P|}A_{soc_G} y_j'\cdot (x)g$. It clear by construction
that every $y'_j\cdot(x)g$ is a sum of monomials of the form $(u)g_{j,1}\cdots (u)g_{j,s}$
with $g_{j,k}\in G$.
\end{proof}

\section{Appendix on adjoint functors}\label{adj_funct}

\begin{df}\label{adj_df}
Let $\A$ and $\B$ be categories.
A pair of covariant functors
$$(L,F)\ with\ L:\ {\A} \longrightarrow {\B}\ and\ F:\ {\B} \longrightarrow {\A}$$
is called an
{\bf adjoint pair} if there is an isomorphism of contra- covariant
bifunctors:
$$ \Psi^{(L,F)}:\ {\B}(\ L(.)\ ,\ .\ )\ \cong \
{\A}(\ .\ ,\ F(.)\ ).$$
In this case $L$ is called a {\bf left adjoint} of $F$, which is
itself called a  {\bf right adjoint} of $L$.
The adjointness of $(L,F)$ induces two morphisms of functors, a
{\bf unit}
$$ u^{(L,F)}:\ Id_{\A} \longrightarrow FL,\ u^{(L,F)}_a\ \ =\ \
\Psi^{(L,F)}({\rm id}_{L(a)}); $$ and a  {\bf counit}
$$ c^{(L,F)}:\ LF \longrightarrow Id_{\B},\ c^{(L,F)}_b\ =\ \
(\Psi^{(L,F)})^{-1}({\rm id}_{F(b)}).$$
\end{df}

If the context is clear, we will freely omit the upper indices ${(L,F)}$.
For $\beta \in {\A}(a_0,a_1)$,$\gamma \in {\B}(b_0,b_1)$,
$a \in {\A}$ and $b \in {\B}$ we have the following commutative diagrams:

\begin{diagram}
\B(L(a_1),b)     &\rTo^\cong&\A(a_1,F(b))&\ &\ &\ &\B(L(a),b_0)&\rTo^\cong&\A(a,F(b_0))\\
\dTo^{L(\beta)^*}&        &\dTo_{\beta^*}&\ &\ &\ &\dTo^{\gamma_*}& &\dTo_{F(\gamma)_*} \\
\B(L(a_0),b)&\rTo^\cong&\A(a_0,F(b))     &\ &\ &\ &\B(L(a),b_1)&\rTo^\cong&\A(a,F(b_1))\\
\end{diagram}

Hence we get for any $\alpha \in {\B}(L(a_1),b)$:
$$\Psi(\alpha) \circ \beta = (\beta)^*(\Psi(\alpha))=
\Psi(L(\beta)^*(\alpha))= \Psi(\alpha\circ L(\beta));$$
and for any $\gamma \in {\B}(b,b_1)$:
$$F(\gamma)\circ \Psi(\alpha)  = F(\gamma)_*( \Psi(\alpha)) =
\Psi(\gamma_*(\alpha)) = \Psi(\gamma \circ \alpha).$$
In particular we see:
$$F(c^{(L,F)}) \circ u^{(L,F)}_F= F(c^{(L,F)}) \circ \Psi({\rm id}_{LF})=
\Psi(c^{(L,F)}) =  Id_F;$$
and similarly: $$\Psi(\ c^{(L,F)} \circ L(u^{(L,F)}) \ ) =
\Psi(c^{(L,F)}) \circ u^{(L,F)} = u^{(L,F)} = \Psi(\id_{L});$$
hence $c^{(L,F)} \circ L(u^{(L,F)}) =  Id_{L}$.
So the compositions of morphisms of functors
$$F(c) \circ u_F: F \longrightarrow FLF \longrightarrow F$$
and
$$c_{L} \circ L(u): L \longrightarrow LFL \longrightarrow L$$
are the respective identity morphisms.
\\[2mm]
We also can recover the bifunctorial isomorphism $\Psi$ and its
inverse from $L,F,u$ and $c$:\\
Indeed for $\alpha \in {\B}(L(a),b)$
and $\beta \in {\A}(a,F(b))$ we get:
$\Psi(\alpha) = \Psi(\alpha \circ {\rm id}) = F(\alpha) \circ \Psi({\rm id}) =
F(\alpha) \circ u$ and
$$\Psi(c \circ L(\beta)) = \Psi(\Psi^{-1}({\rm id})\circ L(\beta)) = {\rm id} \circ \beta.$$
so $\Psi^{-1}(\beta) = c \circ L(\beta)$.\\
Notice also that for any morphisms $\alpha$ in ${\A}$
and $\beta$ in ${\B}$: $L(\alpha) = {\rm id}_{L} \circ L(\alpha) =
\Psi^{-1}(\Psi({\rm id}_{L}) \circ \alpha ) = \Psi^{-1} \circ (u)_* (\alpha);$
and $F(\alpha) = F(\alpha) \circ {\rm id}_F =
\Psi(\alpha \circ \Psi^{-1}({\rm id}_F)) = \Psi \circ (c)^* (\alpha)$.

\bigskip
Let us summarize these formulae
\begin{thm}\label{adj_form} ({\bf Adjunction - formulae}):Suppose the covariant functors
$$L: {\A} \longrightarrow {\B}\ {\rm and}\ F: {\B} \longrightarrow {\A}$$
build an adjoint pair
$(L,F)$ with the isomorphism of bifunctors:
$$ \Psi:\ {\B}(\ L(.)\ ,\ .\ )\ \cong \
{\A}(\ .\ ,\ F(.)\ ),$$
the unit
$$ u:\ Id_{\A} \longrightarrow FL,\ u_a\ \ =\ \
\Psi({\rm id}_{L(a)}), $$ and the counit
$$ c:\ LF \longrightarrow Id_{\B},\ c_b\ =\ \
\Psi^{-1}({\rm id}_{F(b)}).$$
Then the following relations hold whenever they are meaningful:
$$\Psi(\alpha) \circ \beta\ =\  \Psi(\alpha\circ L(\beta));
\ \ F(\gamma)\circ \Psi(\alpha)\ =\ \Psi(\gamma \circ \alpha);\eqno(1.1)$$
$$F(c) \circ u_F \ =\ Id_F;\ \ c \circ L(u) \ =\  Id_{L};\eqno(1.2)$$
$$\Psi\ =\ (u)^* \circ F;\ \ \Psi^{-1}\ =\ (c)_* \circ L;\eqno(1.3)$$
$$F \ =\ \Psi \circ (c)^*;\ \ L \ =\ \Psi^{-1} \circ (u)_*.\eqno(1.4)$$
Suppose $L: \A \to \B$ and $F: \B \to \A$ are functors, then $(L,F)$ is an
adjoint pair if and only if one of the following holds:
\begin{enumerate}
\item[(i)] For any $a\in \A$ and $b\in \B$ there is an isomorphism
$$\Psi: \B(L(a), b) \cong \A(a,R(b))$$ such that $(1.1)$ holds.
\item[(ii)] There are morphisms $u: Id_\A \to FL$ and $c: LF \to Id_\B$ such
that $(1.2)$ holds.
\end{enumerate}
\end{thm}

\begin{proof}
We only have to verify the second part. Consider (i): then $(1.1)$ is just
a way of saying that $\Psi$ is an isomorphism of bifunctors.
Consider (ii): Then define $\Psi, \Psi^{-1}$ by $(1.3)$ and observe:
$$\Psi \Psi^{-1} (\alpha) = u^* F(c L(\alpha)) =
F(c) FL(\alpha) u = F(c) u \alpha = \alpha.$$
$\Psi^{-1} \Psi (\beta) = c_* L(F(\beta) u) = c LF(\beta) L(u)$
$= \beta c L(u) = \beta$. Moreover $\Psi(\alpha) \beta = u^* F(\alpha) \beta =$
$F(\alpha) u \beta = F(\alpha) FL(\beta) u = F(\alpha L(\beta)) u = \Psi(\alpha L(\beta))$
and $F(\gamma) \Psi(\alpha) = F(\gamma) F(\alpha) u = F(\gamma \alpha) u$
$= \Psi(\gamma \alpha)$. Now the adjointness follows from $i)$.
\end{proof}

Let $F:\ \B\to\A$ and $L,R:\ \A\to\B$ be functors such that
$(L,F)$ and $(R,F)$ are adjoint pairs with corresponding isomorphisms
$$\Psi^{(L,F)}:\ \B(L(a),b)\to\A(a,F(b))\ {\rm and}\
\Psi^{(F,R)}:\ \A(F(b),a')\to\B(b,R(a')).$$
Then for $\beta\in\B(L(a),b)$ we have $\Psi^{(L,F)}(\beta)=F(\beta)\circ u^{(L,F)}_a$ and
for $\alpha\in\B(b,R(a'))$ we have $(\Psi^{(F,R)})^{-1}(\alpha)= c^{(F,R)}_{a'}\circ F(\alpha)$,
where $u^{(L,F)}$ and $c^{(F,R)}$ denote unit and co-unit of the corresponding adjoint pairs.
It follows that
$(\Psi^{(F,R)})^{-1}(\alpha)\circ \Psi^{(L,F)}(\beta)=
c^{(F,R)}_{a'}\circ F(\alpha\circ\beta)\circ u^{(L,F)}_a\in \A(a,a').$
This observation gives rise to the following

\begin{df} \label{trace_of_two_sd_adj}\label{ex2_7}
Let $F: {\B} \longrightarrow {\A}$ be a functor with left
and right adjoints ${L},R: {\A} \longrightarrow {\B}$.
We define the $F$-trace as the map
$${\rm T}_F: {\B}({L}(a),R(a')) \longrightarrow
{\A}(a,a')$$
$${\rm T}_F(\alpha) =c^{(F,R)}_{a'}\circ F(\alpha)\circ u^{({L},F)}_a.$$
With $\A(a,a')_{a''}$ we denote the subset of morphisms in
$\A(a,a')$ that factor through the object $a''\in\A$ and we set
$\A(a,a')_F:=\cup_{b\in\B}\A(a,a')_{F(b)}$.
\end{df}

From the preceding discussion and Theorem \ref{adj_form} we obtain:
\begin{lemma}\label{F-trace_formulae}
Let $\beta\in\B(L(a),b)$, $\alpha\in\B(b,R(a'))$, $\gamma\in\A(a'',a)$ and
 $\delta\in\A(a',a'')$. Then we have:
 \begin{enumerate}
\item ${\rm T}_F(\alpha) = (\Psi^{(F,R)} )^{-1}(\alpha) \circ u^{({L},F)} =
c^{(F,R)}\circ \Psi^{(L,F)}(\alpha)$, if $b=R(a')$.
\item ${\rm T}_F(\alpha\circ \beta) = (\Psi^{(F,R)} )^{-1}(\alpha) \circ
\Psi^{({L},F)}(\beta)$;
\item ${\rm T}_F(\alpha)\circ\gamma={\rm T}_F(\alpha\circ L(\gamma))$;
\item $\delta\circ {\rm T}_F(\alpha)={\rm T}_F(R(\delta)\circ\alpha)$.
\end{enumerate}
\end{lemma}

The next result describes the image of the $F$-trace map ${\rm T}_F$:

\begin{prp}\label{image_of_T}
Suppose the situation of \ref{ex2_7}, then: ${\rm T}_F({\B}({L}(a_1),R(a_2))) =
{\A}(a_1,a_2)_F.$
For an object $a \in {\A}$ the following are equivalent:
\begin{enumerate}
\item ${\rm id}_a \in \A(a,a)_F$;
\item ${\rm id}_a\ factors\ through\ F{L}(a)$;
\item ${\rm id}_a\ factors\ through\ FR(a)$;
\item the co-unit $c^{(F,R)}$ splits;
\item the unit $u^{(L,F)}$ splits.
\end{enumerate}
\end{prp}
\begin{proof} Set $\Psi:=\Psi^{({L},F)}$, $\Psi':=\Psi^{(F,R)}$ and let
$\alpha = {\rm T}_F(\beta)$ for $\beta \in
{\B}({L}(a_1),R(a_2))$, then
$$\alpha = c^{(F,R)}\circ
\Psi(\beta)=\Psi'^{-1}(\beta)\circ u^{(L,F)}$$
with $\Psi(\beta)\in {\A}(a_1,FR(a_2))$ and
$\Psi'^{-1}(\beta)\in {\A}(FL(a_1),a_2)$,
hence
$$\alpha \in {\A}(a_1,a_2)_F\cap {\A}(a_1,a_2)_{FL(a_1)}\cap {\A}(a_1,a_2)_{FR(a_2)}.$$
On the other hand ${\rm T}_F(\Psi'({\rm id}_{F(b)}) \circ \Psi^{-1}({\rm id}_{F(b)})) =
{\rm id}\circ {\rm id} = {\rm id}_{F(b)}$. Now the ``two - sided - ideal" property of the
image of $F$-trace - maps, from Lemma \ref{F-trace_formulae} (3) and (4), implies that all morphisms factoring through $F$ belong to this image.
The implications (1) $\Rightarrow$ (2) and (3) follow from the proof above.
Clearly (2) or (3) imply (1) and (1) $\iff$ (4) as well as (1) $\iff$ (5)
follow from Lemma \ref{F-trace_formulae} (1).
This finishes the proof.
\end{proof}

\begin{df}\label{split_objects}
Let $(L,F)$ be an adjoint pair of functors with $L:\ \A\to\B$ and $F:\ \B\to\A$.
An object $a\in\A$ is called ($L$-) split if the unit morphism
$u^{(L,F)}_a:\ a\to FL(a)$ has a left inverse. An object
$b\in\B$ is called ($F$-) split if the co-unit map
$c^{(L,F)}_b:\ LF(b)\to b$ has a right inverse.
\end{df}

The following results will turn out to be useful:
\begin{prp}\label{s-prj_cat}
Let $(L,F)$ be as in Theorem \ref{adj_form} and assume that $F$ respects epimorphisms\footnote{
for example, if $F$ has a left- and a right adjoint}
(surjective maps). Then the following hold:
\begin{enumerate}
\item If $a\in\A$ is (s-)projective, then $L(a)$ is (s-)projective.
\item If $b\in\B$ is $F$-split and $F(b)\in\A$ is (s-)projective, then $b$ is (s-)projective.
\end{enumerate}
\end{prp}
\begin{proof}
(1):\ Let $\alpha\in\B(b_1,b_2)$ be an epimorphism (surjective morphism) and
$\beta\in\B(L(a),b_2)$. Then since $a$ is (s-) projective there is $\gamma\in\A(a,F(b_1))$
with $F(\alpha)\circ\gamma=F(\beta)\circ u^{(L,F)}_a$.
\begin{diagram}\label{gamma_diagram1}
a&\rTo^{u^{(L,F)}_a}                     &FL(a)\\
\dDashto^{\gamma}&\ &\dTo_{F(\beta)}\\
F(b_1)&\rOnto_{F(\alpha)}&F(b_2)\\
\end{diagram}
Consider the commutative diagram:
\begin{diagram}\label{Psi_diagram}
\A(a,F(b_1))&\rTo^{\Psi^{-1}}&\B(L(a),b_1)\\
\dTo^{F(\alpha)_*}&\ &\dTo_{\alpha_*}\\
\A(a,F(b_2))&\rTo^{\Psi^{-1}}&\B(L(a),b_2).\\
\end{diagram}
Set $\Gamma:=\Psi^{-1}(\gamma)\in\B(L(a),b_1)$; then
$\alpha\circ\Gamma=$ $\alpha_*\circ\Psi^{-1}(\gamma)=$
$\Psi^{-1}\circ F(\alpha)_*(\gamma)=$
$\Psi^{-1}(F(\beta)\circ u^{(L,F)}_a)=\beta$,
since $F(\beta)\circ u^{(L,F)}=(u^{(L,F)*}\circ F)(\beta)=\Psi(\beta)$.
This shows that $L(a)$ is (s-) projective.\\
(2):\ Let $\alpha\in\B(b,b_2)$ and $\beta\in\B(b_1,b_2)$ an epimorphism (surjective map).
Since $F(b)$ is (s-)projective there is $\gamma\in\A$ completing the diagram
\begin{diagram}\label{gamma_diagram2}
 &                     &F(b_1)\\
 &\ruDashto^{\gamma} &\dOnto_{F(\beta)}\\
F(b)&\rTo_{F(\alpha)}&F(b_2)\\
\end{diagram}
Let $s\in\B(b,LF(b))$ be a right inverse to $c^{(L,F)}_b$, then
applying $(\Psi^{(L,F)})^{-1}$ gives the following commutative diagram:
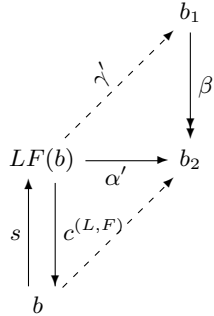
\begin{diagram}\label{psi_inv_gamma_diagram}
 &                     &b_1\\
 &\ruDashto^{\gamma'} &\dOnto_{\beta}\\
LF(b)&\rTo_{\alpha'}&b_2\\
\uTo^s\dTo_{c^{(L,F)}}&\ruDashto&\\
b\ &\ &
\end{diagram}
with $\gamma':=(\Psi^{(L,F)})^{-1}(\gamma)$ and $\alpha':=(\Psi^{(L,F)})^{-1}(F(\alpha))$.
Indeed, using the formulae of Theorem \ref{adj_form} we get
$\alpha'\circ s=$ $c^{(L,F)}\circ LF(\alpha)=\alpha\circ c^{(L,F)}\circ s=\alpha$.
This shows that $b$ (s-) projective.
\end{proof}

The following is a straightforward dualization of Proposition \ref{s-prj_cat}
\begin{prp}\label{i-inj_cat}
Let $(L,F)$ be as in Theorem \ref{adj_form} and assume that $L$ respects monomorphisms \footnote{
for example, if $F$ has a left- and a right adjoint}
(injective maps). Then the following hold:
\begin{enumerate}
\item If $b\in\B$ is (i-)injective, then $F(b)$ is (i-)injective.
\item If $a\in\A$ is $L$-split and $L(a)\in\B$ is (i-)injective, then $a$ is (i-)injective.
\end{enumerate}
\end{prp}

\bibliographystyle{plain}

\end{document}